\documentclass[a4paper,oneside,10pt]{article}
\usepackage{amsmath}
\usepackage{amssymb}
\usepackage{amsfonts}
\usepackage{graphicx}
\usepackage{color}
\usepackage{amsthm}
\usepackage{cite}
\usepackage{enumitem}
\usepackage{hyperref} 
\usepackage{bbm}
\usepackage[utf8]{inputenc}

\newtheoremstyle{uprightstyle}
{3pt} 
{3pt} 
{\upshape} 
{} 
{\bfseries} 
{.} 
{ }
{}

\theoremstyle{uprightstyle}

\newtheorem{theorem}{Theorem}[section]

\newtheorem{definition}[theorem]{Definition}

\newtheorem{example}[theorem]{Example}

\newtheorem{lemma}[theorem]{Lemma}

\newtheorem{proposition}[theorem]{Proposition}
\newtheorem{remark}[theorem]{Remark}

\newenvironment{proof}[1][Proof]{\noindent \textbf{#1.} }{\ \ $\Box$}

\numberwithin{equation}{section}

\begin{document}
	\title{Stochastic Volterra integral equations driven by $ G $-Brownian motion}
\author{Bingru Zhao \thanks{Zhongtai Securities Institute for Financial Studies,
		Shandong University, Jinan, Shandong 250100, PR China. bingruzhao@mail.sdu.edu.cn.}
	\and Renxing Li\thanks{Zhongtai Securities Institute for Financial Studies,
		Shandong University, Jinan, Shandong 250100, PR China. 202011963@mail.sdu.edu.cn.}
		\and Mingshang Hu \thanks{Zhongtai Securities Institute for Financial Studies,
		Shandong University, Jinan, Shandong 250100, PR China. humingshang@sdu.edu.cn.
		Research supported by the National Natural Science Foundation of China (No. 12326603, 11671231). }
}
\makeatletter
\renewcommand{\@date}{} 
\makeatother
	\maketitle

	\textbf{Abstract}. 
	In this paper, we study the stochastic Volterra integral equation driven by $G$-Brownian motion ($G$-SVIE). The existence, uniqueness and two types of continuity of the solution to $G$-SVIE are obtained. Moreover, combining a new quasilinearization technique with the two-step approximation method, we establish the corresponding comparison theorem for a class of $G$-SVIEs. In particular, by means of this method, the classical assumptions on partial derivatives of the coefficients are unnecessary.
	
	
	
	\addcontentsline{toc}{section}{\hspace*{1.8em}Abstract}
	\section{Introduction}
	\makeatletter
	\newcommand{\rmnum}[1]{\romannumeral #1}
	\newcommand{\Rmnum}[1]{\expandafter\@slowromancap\romannumeral #1@}
	\makeatother
	\noindent
	
	Let $ B $ be a standard Brownian motion on the Wiener space. A classical stochastic Volterra integral equation (SVIE) takes the following form
	\begin{equation}\label{introduction1}
		X(t)=\phi(t)+\int_{0}^{t} b(t, s, X(s)) ds +\int_{0}^{t} \sigma(t, s, X(s)) dB_s, \quad t\in [0,T].
	\end{equation}
	It is observed that SVIE is a natural extension of stochastic differential equation (SDE). Many researchers are devoted to the theory of SVIE. The existence and uniqueness of the solution to the SVIE was first studied in \cite{berger1980volterra,ito1979on}, while the SVIE driven by a semimartingale was established by Protter \cite{protter1985volterra}. Relevant studies on comparison theorems can refer to \cite{tudor1989a,ferreyra2000comparison,wang2015com}. In practical applications, noise often displays long-range dependence, and state equations may involve memory properties. Thus, SVIE is inherently suited to practical systems, serving as an important tool in stochastic control (see, e.g.,\cite{wang2018lq,shi2015optimal,wangh2023linear}). For more related works, see \cite{pardoux1990stochastic,W. George1995sigular,yong2008well,yong2006backward,wang2017optimal,oksendal1993stochastic,zhang2010stochastic} and the references therein.
	
	Inspired by model uncertainty, Peng \cite{peng2005nonlinear,peng2007G,peng2019nonlinear,peng2004filtration,peng2008multi,li2011stopping} systematically proposed $G$-expectation theory, which departs from the constraints of classical probability frameworks. SDEs driven by $ G $-Brownian motion ($ G $-SDEs) are widely studied. Gao \cite{gao2009pathwise} investigated pathwise properties and homeomorphic  properties with respect to the initial values for $ G $-SDEs. Bai and Lin \cite{linbai2014} established the existence and uniqueness of the solution to the $G$-SDE with integral-Lipschitz coefficients. More related works can be found in \cite{luo2014stochastic,linyiqing2013rgsde,lin2013some,gao2010large} and the references therein. 
	
	In this paper, we study a class of stochastic Volterra integral equations driven by a $G$-Brownian motion ($ G $-SVIEs) of the form
	\begin{equation}\label{introdution2}
		X(t)=\phi(t)+\int_{0}^{t} b(t, s, X(s)) ds+\int_{0}^{t} h(t, s, X(s)) d\left\langle B\right\rangle_s +\int_{0}^{t} \sigma(t, s, X(s)) dB_s, \quad t \in[0,T],
	\end{equation}
	where the coefficients are of linear growth and Lipschitz continuous in $x$. Since $b$, $h$, $\sigma$ depend on $(t,s)$, we cannot know whether  the integral $\int_{0}^{t} b(t, s, X(s)) ds+\int_{0}^{t} h(t, s, X(s)) d\left\langle B\right\rangle_s +\int_{0}^{t} \sigma(t, s, X(s)) dB_s $, $t\in[0,T]$ belongs to $ M_{G}^{2}(0,T)$ for each $X(\cdot)\in M_{G}^{2}(0,T)$, which is important in constructing the map \eqref{3.2}. To overcome this difficulty, we impose a modulus of continuity assumption with respect to $ t $ for $b$, $h$, $\sigma$, ensuring that the map \eqref{3.2} is well-defined. Then, by using the contraction mapping principle, we prove the existence and uniqueness of the solution to $G$-SVIE \eqref{introdution2}. Moreover, the mean-square continuity and pathwise continuity of solution are established.
	
	Observe that the comparison theorem for $ G $-SVIEs with general Volterra-type diffusion coefficients cannot hold even when $ G $ is a linear function (see \cite{tudor1989a}). Therefore, we consider the case where  the diffusion coefficient  $ \sigma(t, s, x)$ takes the form $H(t)\sigma(s, x) $. There are two main difficulties in the proof of the comparison theorem. On the one hand, our proof relies on stopping times, but the indicator functions generated by a sequence of stopping times may not belong to $ M_{G}^{2}(0,T) $. Thus, we construct the sequence of stopping times under each $ P\in\mathcal{P} $ by applying the representation theorem of $ G $-expectation (see Theorem \ref{thm2.1}).  On the other hand, the classical quasilinearization technique \eqref{classical_linear} requires the modulus of continuity assumption for the coefficients of the partial derivatives $ \partial_{x}b$ and $ \partial_{x}h $. To relax this assumption, we adopt a new quasilinearization technique (see \eqref{def_n}) and construct two approximation sequences (see \eqref{8}, \eqref{9}). Based on the result of a two-step approximation (see \eqref{two_approximation}), we establish the  comparison theorem for $ G $-SVIEs with the Volterra-type diffusion coefficient $H(t)\sigma(s, x) $. With the help of this approach, the modulus of continuity and Lipschitz continuity assumptions on partial derivatives of coefficients are unnecessarily restrictive, which is different from the assumptions in \cite{tudor1989a,ferreyra2000comparison,wang2015com}. Furthurmore, we can deal with cases where partial derivatives do not exist. Finally, we provide a concrete example in the linear case to demonstrate the effectiveness of our results.
	
	The paper is organized as follows. In section 2, we present some preliminaries for $ G $-expectation framework. In section 3, we prove the existence and uniqueness of solution to the $ G $-SVIE. In section 4, we establish the comparison theorem for $ G$-SVIEs with the Volterra-type diffusion coefficients $H(t)\sigma(s, x) $.
	\section{Preliminaries}
	\noindent
	
	In this section, we briefly recall some basic notions and results of the $ G $-expectation and related $ G $-stochastic analysis. More details can be found in \cite{peng2004filtration,peng2005nonlinear,peng2007G,peng2008multi,peng2019nonlinear} .
	
	Let $\Omega = C_0(\mathbb{R}^+)$ denote the space of all $\mathbb{R}-$valued continuous paths $(\omega_t)_{t\in\mathbb{R}^+}$, with $\omega_0 = 0$, equipped with the distance
	\[ \rho(\omega^{1},\omega^{2}):=\sum_{i = 1}^{\infty}2^{-i}[(\max_{t\in[0,i]}|\omega_t^{1}-\omega_t^{2}|)\wedge 1],\quad \omega^{1},\omega^{2}\in\Omega. \]
	The corresponding canonical process is $B_t(\omega)=\omega_t$, $ t\in[0,\infty) $. For each given $t\geq0$, we define
	\[Lip(\Omega_t):=\{\varphi(B_{t_1\wedge t},\cdots,B_{t_n\wedge t}):n\in\mathbb{N},t_1,\cdots,t_n\in[0,\infty),\varphi\in C_{l.Lip}(\mathbb{R}^{ n})\},\ Lip(\Omega):=\cup_{n=1}^{\infty}Lip(\Omega_{n}), \]
	where $ C_{l,Lip}(\mathbb{R}^{ n}) $ is the linear space of all local Lipschitz functions on $ \mathbb{R}^{ n} $.

	Let $ G:\mathbb{R}\rightarrow\mathbb{R} $ be a given monotonic, sublinear function, i.e., for $ a\in\mathbb{R} $, 
	\begin{align}\label{G}
		G(a)=\frac{1}{2}(\bar{\sigma}^{2}a^{+}-\underline{\sigma}^{2}a^{-}),
	\end{align}
	where $ 0<\underline{\sigma}\leq\bar{\sigma}<\infty  $.
	For each given function $ G$, Peng constructed the $ G $-expectation $ \hat{\mathbb{E}} $ on $ (\Omega,Lip(\Omega)) $(see \cite{peng2019nonlinear} for definition). Then the canonical process $ B $ is a one-dimensional $ G $-Brownian motion under $ \hat{\mathbb{E}} $.
	
	For each $ p\geq1 $, we denote by $ L_{G}^{p}(\Omega) $ the completion of $ Lip(\Omega) $ under the norm $ \|X\|_{L_{G}^{p}}=(\hat{\mathbb{E}}[|X|^{p}])^{\frac{1}{p}} $. Similarly, we can define $ L_{G}^{p}(\Omega_{t}) $ for each fixed $ t $. $ \hat{\mathbb{E}}:Lip(\Omega)\rightarrow\mathbb{R} $ can be continuously extended to the mapping from $ L_{G}^{1}(\Omega) $ to $ \mathbb{R} $. The sublinear expectation space $(\Omega ,L_{G}^{1}(\Omega),\hat{\mathbb{E}})  $ is called $ G $-expectation space.
	
	Define $ \mathcal{F}_{t}:=\sigma(B_{s}:s\leq t)$ and $ \mathcal{F}:=\bigvee_{t\geq0}\mathcal{F}_{t} $. The following is the representation theorem.
	
	\begin{theorem}[\cite{hu2009representation,denis2011function}]\label{thm2.1}
		Let $(\Omega ,L_{G}^{1}(\Omega),\hat{\mathbb{E}})  $ be a $ G $-expectation space. Then there exists a weakly compact set of probability measures $ \mathcal{P} $ on $ (\Omega,\mathcal{F}) $ such that
		\begin{align}
			\hat{\mathbb{E}}[\xi]=\sup_{\mathbb{P}\in\mathcal{P}}E_{P}[\xi], \quad \forall \xi\in L_{G}^{1}(\Omega).
		\end{align} 
	\end{theorem}
	
	From \cite{denis2011function}, we introduce the definition of capacity as follows:
	\[\mathrm{c}(A):=\sup_{\mathbb{P}\in\mathcal{P}}P(A), \quad A\in\mathcal{F}.\]

	\begin{definition}
		A set $ A $ is polar if $ \mathrm{c}(A)=0 $ and a property holds “quasi-surely” (q.s.) if it holds outside a polar set.
	\end{definition}
	
	\begin{definition}
		For each $ T>0 $ and $ p\geq1 $, set
		\[
		M_G^{p,0}(0,T):=\left\{\eta(t)=\sum_{j = 0}^{N - 1}\xi_jI_{[t_j,t_{j + 1})}(t):N\in\mathbb{N},0= t_0< t_1<\cdots< t_N= T, \ \xi_j\in L_G^p(\Omega_{t_j})\right\}.
		\]
		We denote by $ M_{G}^{p}(0,T) $ the completion of $ M_{G}^{p,0}(0,T) $ under the norm $ \|\eta\|_{M_{G}^{p}}:=(\hat{\mathbb{E}}[\int_{0}^{T}|\eta(t)|^{p}dt])^{\frac{1}{p}} $.
	\end{definition}
	
	According to \cite{li2011stopping,peng2007G,peng2008multi,peng2019nonlinear}, the integrals $ \int_{0}^{t}\eta(s)dB_{s}$ and $ \int_{0}^{t}\xi(s)d\langle B\rangle_{s} $ are well-defined for $ \eta(\cdot)\in M_{G}^{2}(0,T) $ and $ \xi(\cdot)\in M_{G}^{1}(0,T) $, where $\langle B\rangle $ denotes the quadratic variation process of $ B $. By Proposition 3.4.5 and Corollary 3.5.5 in \cite{peng2019nonlinear}, we have 
	\[ \underline{\sigma}^{2}dt\leq d\langle B\rangle_{t}\leq\bar{\sigma}^{2}dt, \quad \text{q.s.,} \]
	and for $\eta(\cdot)\in M_{G}^{2}(0,T) $,
	\begin{align}\label{ito}
		\hat{\mathbb{E}}\left[ \left|\int_{0}^{t}\eta(s)dB_{s} \right|^{2} \right] =\hat{\mathbb{E}}\left[\int_{0}^{t}|\eta(s)|^{2}d\langle B\rangle_{s}\right] \leq\bar{\sigma}^{2} \hat{\mathbb{E}}\left[\int_{0}^{t}|\eta(s)|^{2}ds\right].
	\end{align}
	
	\begin{theorem}[\cite{hu2014backward}]
		Let $ p\geq2 $ and $ T>0 $. For all $ \eta(\cdot)\in M_{G}^{p}(0,T) $, we have
		\begin{align}\label{bdg} \underline{\sigma}^{p}c_{p}\hat{\mathbb{E}}\left[\left(\int_{0}^{T}|\eta(s)|^{2}ds \right)^{\frac{p}{2}}  \right]\leq\hat{\mathbb{E}}\left[ \sup_{t\in[0,T]}\left|\int_{0}^{t}\eta(s)dB_{s} \right|^{p} \right] \leq\bar{\sigma}^{p}C_{p} \hat{\mathbb{E}}\left[\left(\int_{0}^{T}|\eta(s)|^{2}ds \right)^{\frac{p}{2}}  \right], \end{align}
		where $ 0<c_{p}<C_{p}<\infty $ are constants.
	\end{theorem}

	For each $ T>0 $ and $ p\geq1 $, define
	\[ \tilde{M}_{G}^{p}(0,T):=\left\lbrace  X(\cdot)\in M_{G}^{p}(0,T): X(t) \in  L_{G}^{p}(\Omega_{t})  \text{ for each } t\in[0,T] \right\rbrace . \] 
	
	\begin{definition}
		A one-dimensional process $  X(\cdot)\in \tilde{M}_{G}^{2}(0,T) $ is said to be mean-square continuous if for any $ t', t\in[0,T] $,
		\[\lim_{t'\rightarrow t}\hat{\mathbb{E}}[|X(t)-X(t')|^{2}]=0. \]
	\end{definition} 
	\begin{theorem}[\cite{denis2011function}]\label{continuous}
		Let $  X(\cdot)\in \tilde{M}_{G}^{p}(0,T) $ for $ T>0 $ and $ p\geq1 $. Assume that there exists positive constants $ c $ and $ \varepsilon $ such that
		\[
		\mathbb{E}[|X(t) - X(s)|^p]\leq c|t - s|^{1+\varepsilon}.
		\]
		Then $ X(\cdot) $ admits a modification $ \tilde{X}(\cdot) $ such that
		\[
		\mathbb{E}\left[\left(\sup_{s\neq t}\frac{|\tilde{X}(t)-\tilde{X}(s)|}{|t - s|^{\alpha}}\right)^p\right]<\infty,
		\]
		for every $ \alpha\in[0,\varepsilon/p) $. As a consequence, paths of $ \tilde{X}(\cdot) $ are quasi-surely H\"{o}lder continuous of order $ \alpha $ for every $ \alpha<\varepsilon/p $. Moreover, $\tilde{X}(\cdot)\in \tilde{M}_{G}^{p}(0,T) $.
	\end{theorem}
	
	\section{Existence and uniqueness of the solution of G-SVIE}
	\noindent
	
	Set $ \Delta = \{(t,s) \in [0,T]^2 \mid 0 \leq s \leq t \leq T\} $. Consider the following stochastic Volterra  integral equation driven by $ G $-Brownian motion ($ G $-SVIE) on the finite interval [0,T]:
	\begin{equation}\label{G-SVIE}
		X(t)=\phi(t)+\int_{0}^{t} b(t, s, X(s)) ds+\int_{0}^{t} h(t, s, X(s)) d\left\langle B\right\rangle_s+\int_{0}^{t} \sigma(t, s, X(s)) dB_s, \quad 0 \leq t \leq T,
	\end{equation}
	where the coefficients  
	\begin{equation*}
		b(t, s,\omega, x),\text{ }h(t,s,\omega,x), \sigma(t,s,\omega,x):\Delta  \times \Omega \times \mathbb{R}
		\rightarrow \mathbb{R}.
	\end{equation*}
	
	We need the following assumptions:
	\begin{description}
		\item[(H1)]For each $ t\in[0,T] $ and $ x\in \mathbb{R},\ b(t,\cdot,x) ,\ h(t,\cdot,x) ,\ \sigma(t,\cdot,x) \in M_{G}^{2}(0,t)$;
		\item[(H2)]For each $ x, \ y\in\mathbb{R} $ and $ t,\ s\in\Delta $, there exists a constant $ L>0 $ such that \[|b(t,s,x)-b(t,s,y)|+|h(t,s,x)-h(t,s,y)|+|\sigma(t,s,x)-\sigma(t,s,y)|\leq L|x-y|,  \]   
		\[|b(t,s,x)|^{2}+|h(t,s,x)|^{2}+|\sigma(t,s,x)|^{2}\leq L^{2}(1+|x|^{2});  \] 		
		\item[(H3)] For each $ x, \ y\in\mathbb{R} $ and each $ (t,s),\ (t', s)\in\Delta$, 
		\[ |b(t',s,x)-b(t,s,x)|+|h(t',s,x)-h(t,s,x)|+|\sigma(t',s,x)-\sigma(t,s,x)|\leq \rho(|t'-t|), \]
		where $ \rho:[0,\infty)\rightarrow [0,\infty)$ is continuous and strictly increasing with $ \rho(0)=0 $.
	\end{description}

	\begin{definition}
		Suppose that $  \phi(\cdot)\in M_{G}^{2}(0,T)  $ and $ b$,  $h$, $\sigma $ satisfy the assumptions (H1)-(H3). The process $ X(\cdot) $ is called the solution of the equation \eqref{G-SVIE} if it satisfies the following properties:\\
		(1)$ X(\cdot) \in M_{G}^{2}(0,T)$;\\
		(2)$ X(t)=\phi(t)+\int_{0}^{t} b(t, s, X(s)) ds+\int_{0}^{t} h(t, s, X(s)) d\left\langle B\right\rangle_s+\int_{0}^{t} \sigma(t, s, X(s)) dB_s, \quad 0 \leq t \leq T $.
	\end{definition}
	
	Unlike the general $G$-SDE,  the coefficients of $G$-SVIE depend on both $ t $ and $ s $. Therefore,  we get the following Lemma.
	
	\begin{lemma}\label{MG}
		Let $ X(\cdot)\in M_{G}^{2}(0,T) $. Suppose that $ b, h, \sigma $ satisfy the assumptions (H1)-(H3). Then $M(\cdot)\in \tilde{M}_{G}^{2}(0,T)$, where
		\[M(t)=\int_{0}^{t}\sigma(t,s,X(s))dB_s. \]
		Similarly, $N(\cdot)\in \tilde{M}_{G}^{2}(0,T)$, where  \[ N(t)=\int_{0}^{t}b(t, s, X(s)) ds+\int_{0}^{t} h(t, s, X(s)) d\left\langle B\right\rangle_s .\]
	\end{lemma}   
	\begin{proof}[Proof]
		Since $ \sigma(t,\cdot,x)\in M_{G}^{2}(0,T) $ for each given $ t,x $, it follows that $ M(t)\in L_{G}^{2}(\Omega_{t}) $. Let $ \Pi_{T}^n=\{0=t_{0}^n<\cdots<t_{n}^n=T  \} $ be a sequence of partitions of $ [0, T] $, where \[ \|\Pi_{T}^{n}\|:=\max\{t_{i+1}^n-t_{i}^n,0\leq i\leq n-1 \}\rightarrow0 ,\quad \text{as }  n\rightarrow\infty.\]  Define 
		\[ M_{n}(t)=\sum_{i=0}^{n-1}M(t_{i}^n)I_{[t_{i}^n,t_{i+1}^n)}(t), \quad\ t\in[0,T].\] It is clear that $ M_{n}(\cdot)\in M_{G}^{2}(0,T) $. Therefore, using (H2)-(H3) and \eqref{ito}, we derive that for each $t\in[t_i^n,t_{i+1}^n]$,
		\begin{equation}\label{3.2_1}
			\begin{aligned}
				&\hat{\mathbb{E}}[|M(t)-M(t_i^n)|^{2}]
				\\\leq&\hat{\mathbb{E}}\left[\left|\int_{0}^{t_{i}^n}(\sigma(t,s,X(s))-\sigma(t_{i}^n,s,X(s)))dB_s+\int_{t_{i}^n}^{t}\sigma(t,s,X(s))dB_s\right|^2\right]
				\\\leq&2\bar{\sigma}^2\left( \hat{\mathbb{E}}\left[\int_{0}^{t_{i}^n}\rho^{2}(|t-t_{i}^n|)ds\right]+\hat{\mathbb{E}}\left[\int_{t_{i}^n}^{t}L^2\left(1+|X(s)|^{2}\right)I_{\{|X(s)|\leq N\}}ds\right]\right.
				\\&\left.+\hat{\mathbb{E}}\left[\int_{t_{i}^n}^{t}L^2\left(1+|X(s)|^{2}\right)I_{\{|X(s)|\geq N\}}ds\right]\right) 
				\\\leq &2\bar{\sigma}^2\left( \rho^{2}(\|\Pi_{T}^{n}\|)T+L^2\|\Pi_{T}^{n}\|(1+N^{2})+L^2\hat{\mathbb{E}}\left[\int_{0}^{T}(1+|X(s)|^{2})I_{\{|X(s)|\geq N\}}ds\right] \right) .
			\end{aligned}
		\end{equation}
		From \eqref{3.2_1}, we conclude that 
		\begin{equation}\label{3.2_2}
			\begin{aligned}
				&\sup\limits_{t\in[0,T]}\hat{\mathbb{E}}[|M(t)-M_{n}(t)|^{2}]
				\\\leq&\sum_{i=0}^{n-1}\hat{\mathbb{E}}\left[\left|M(t)-M(t_{i}^n)\right|^{2} \right]I_{[t_{i}^n,t_{i+1}^n)}(t)
				\\\leq{}&2\bar{\sigma}^2\left( \rho^{2}(\|\Pi_{T}^{n}\|)T+L^2\|\Pi_{T}^{n}\|(1+N^{2})+L^2\hat{\mathbb{E}}\left[\int_{0}^{T}(1+|X(s)|^{2})I_{\{|X(s)|\geq N\}}ds\right] \right) .
			\end{aligned}
		\end{equation}
		By Theorem 4.7 in \cite{hu2016quasi}, we have $ \lim\limits_{N\rightarrow\infty}\hat{\mathbb{E}}[\int_{0}^{T}(1+|X(s)|^{2})I_{\{|X(s)|\geq N\}}ds] =0$. Sending $n\rightarrow\infty $, and then $N\rightarrow\infty$,
		we derive $ \sup\limits_{t\in[0,T]}\hat{\mathbb{E}}[|M(t)-M_{n}(t)|^{2}]\rightarrow0  $ and
		\[ \lim_{n\rightarrow\infty}\hat{\mathbb{E}}\left[\int_{0}^{T}|M(t)-M_{n}(t)|^{2}dt\right]\leq \lim_{n\rightarrow\infty} T\sup\limits_{t\in[0,T]}\hat{\mathbb{E}}\left[|M(t)-M_{n}(t)|^{2}\right]=0. \]
		It follows that $ M(\cdot)\in \tilde{M}_{G}^{2}(0,T)$. The proof for $N(\cdot)$ can be established in a similar way. 
	\end{proof}
	\begin{remark}\label{well-def}
		The assumption (H3) ensures that the integral  \[M(t)+N(t)=\int_{0}^{t}b(t,s,X(s))ds+\int_{0}^{t}h(t,s,X(s))d\langle B\rangle_s+\int_{0}^{t}\sigma(t,s,X(s))dB_s,\text{ } t\in[0,T] \] belongs to $\tilde{M}_{G}^{2}(0,T)$ for each $X(\cdot)\in M_{G}^{2}(0,T)$. Then the integral $ M(\cdot)+N(\cdot)$ is well-defined.
	\end{remark}
	\begin{lemma}\label{sup}
		Let $ X(\cdot)\in \tilde{M}_{G}^{2}(0,T)$ be mean-square continuous. Then 
		\[ \sup_{t\in[0,T]}\hat{\mathbb{E}}[|X(t)|^{2}]<\infty. \]
	\end{lemma}
	\begin{proof}
		Since $ X (\cdot)$ is mean-square continuous, there exists a constant $ \delta>0 $ such that for each $ t',\ t\in[0,T] $ with $ |t'-t|<\delta $,
		\[ \hat{\mathbb{E}}[|X(t')-X(t)|^{2}]<1 .\]  Let $ \Pi_{T}^n=\{0=t_{0}^n<\cdots<t_{n}^n=T  \} $ be a sequence of partitions of $ [0, T] $, where \[ \|\Pi_{T}^{n}\|:=\max\{t_{i+1}^n-t_{i}^n,0\leq i\leq n-1 \}<\delta .\] For each $ t\in[0,T]$, there exists $i$ such that $t\in [t_{i}^n,t_{i+1}^n) $. Thus we have
		\begin{align*}
			\hat{\mathbb{E}}[|X(t)|^{2}]={}&\hat{\mathbb{E}}[|X(t)-X(t_{i}^n)+X(t_{i}^n)|^{2}]\\
			\leq{}&2\left(\hat{\mathbb{E}}[|X(t)-X(t_{i}^n)|^{2}]+\hat{\mathbb{E}}[|X(t_{i}^n)|^{2}] \right)\\
			\leq{}&2\left(1+\hat{\mathbb{E}}[|X(t_{i}^n)|^{2}]\right),
		\end{align*}
		which implies that \[\sup\limits_{t\in[0,T]}\hat{\mathbb{E}}[|X(t)|^{2}]<\infty. \] 
	\end{proof}
	
	Now we prove the existence and uniqueness of the solution of equation (\ref{G-SVIE}), which is the main result in this section.
	\begin{theorem}\label{jie1}
		Suppose that $  \phi(\cdot)\in M_{G}^{2}(0,T)  $ and $ b$, $h$, $ \sigma $ satisfy the assumptions (H1)-(H3). Then equation (\ref{G-SVIE}) admits a unique solution $ X(\cdot)\in M_{G}^{2}(0,T) $. Moreover, if $  \phi(\cdot)\in \tilde{M}_{G}^{2}(0,T)  $ is mean-square continuous, then $ X(\cdot)\in \tilde{M}_{G}^{2}(0,T) $ is mean-square continuous.
	\end{theorem}
	\begin{proof}[Proof]
		\textbf{Existence and uniqueness}.
		Define a map $ \Lambda:M_{G}^{2}(0,T)\mapsto M_{G}^{2}(0,T)$ as follows: for each $ X\in M_{G}^{2}(0,T) $,
		\begin{equation}\label{3.2}
			\Lambda_{t}(X)=\phi(t)+\int_{0}^{t} b(t, s, X(s)) ds+\int_{0}^{t} h(t, s, X(s)) d\left\langle B\right\rangle_s +\int_{0}^{t} \sigma(t, s, X(s)) dB_s,\quad t\in[0,T]
		\end{equation}
		According to Lemma \ref{MG}, we obtain that the map $ \Lambda$ is well-defined. 
		
		Next, we will prove that $ \Lambda$ is a contraction.  For each $ X(\cdot),Y(\cdot) \in M_{G}^{2}(0,T) $, we get
		\begin{equation}\label{3.4}
			\Lambda_{t}(X)-	\Lambda_{t}(Y)=\int_{0}^{t}\tilde{b}(t,s)ds+\int_{0}^{t}\tilde{h}(t,s)d\langle B\rangle_s+\int_{0}^{t}\tilde{\sigma}(t,s)dB_s,
		\end{equation}
		where
		\[ \tilde{b}(t,s)=b(t,s,X(s))-b(t,s,Y(s)),\ \tilde{h}(t,s)=h(t,s,X(s))-h(t,s,Y(s)),\ \tilde{\sigma}(t,s)=\sigma(t,s,X(s))-\sigma(t,s,Y(s)). \]
		Noting that 
		\[|\tilde{b}(t,s)|\leq L|X(s)-Y(s)|.
		\]
		And $ \tilde{h}(t,s), \tilde{\sigma}(t,s)$ also hold. Thus we have
		\begin{equation}\label{3.3}
			\begin{aligned}
				\left|\Lambda_{t}(X)-\Lambda_{t}(Y)\right|^2&\leq3\left\{\left(\int_{0}^{t}\tilde{b}(t,s)ds\right)^2+\bar{\sigma}^4\left(\int_{0}^{t}\tilde{h}(t,s)ds\right)^2+\left|\int_{0}^{t}\tilde{\sigma}(t,s)dB_s\right|^2\right\}
				\\&\leq3(T+\bar{\sigma}^4T)L^2\int_{0}^{t}\left|X(s)-Y(s)\right|^2ds+3\left|\int_{0}^{t}\tilde{\sigma}(t,s)dB_s\right|^2.
			\end{aligned}
		\end{equation}
		Set $\beta >0$, it follows that
		\begin{equation}\label{3.5}
			\begin{aligned}
				&\hat{\mathbb{E}}\left[\int_{0}^{T}e^{-\beta t}\left|\Lambda_{t}(X)-\Lambda_{t}(Y)\right|^2dt\right]
				\\\leq&3\left\{(T+\bar{\sigma}^4T)L^2\hat{\mathbb{E}}\left[\int_{0}^{T}e^{-\beta t}\int_{0}^{t}\left|X(s)-Y(s)\right|^2dsdt\right]+\hat{\mathbb{E}}\left[\int_{0}^{T}e^{-\beta t}\left|\int_{0}^{t}\tilde{\sigma}(t,s)dB_s\right|^2dt\right]\right\}.
			\end{aligned}
		\end{equation}
		For each fixed $t\in[0,T]$, set \[
		L(t,r)=\int_{0}^{r}\tilde{\sigma}(t,s)dB_s,\quad r\in[0,t],\ (t,s)\in \Delta.
		\]
		Applying the $ \text { Itô's} $ formula to $ |L(t,r)|^2$ for $r$ on [0,t] and setting $r=t$, we have
		\begin{equation}\label{4.9_1}
			\begin{aligned}
				\left|\int_{0}^{t}\tilde{\sigma}(t,s)dB_s\right| ^2=2\int_{0}^{t}L(t,s)\tilde{\sigma}(t,s)dB_s+\int_{0}^{t}|\tilde{\sigma}(t,s)|^2d\langle B\rangle_s.
			\end{aligned}
		\end{equation}
		Then, we derive
		\begin{equation}\label{4.9_2}
			\begin{aligned}
				\left|\int_{0}^{t}\tilde{\sigma}(t,s)dB_s\right| ^2-\bar{\sigma}^2\int_{0}^{t}|\tilde{\sigma}(t,s)|^2ds=2\int_{0}^{t}L(t,s)\tilde{\sigma}(t,s)dB_s+\int_{0}^{t}|\tilde{\sigma}(t,s)|^2d(\langle B\rangle_s-\bar{\sigma}^2s),
			\end{aligned}
		\end{equation}
		which together with Proposition 4.1.4 in \cite{peng2019nonlinear} implies that
		\begin{equation}\label{4.9_3}
			\begin{aligned}
				\hat{\mathbb{E}}\left[\left(\left|\int_{0}^{t}\tilde{\sigma}(t,s)dB_s\right|^2-\bar{\sigma}^2\int_{0}^{t}|\tilde{\sigma}(t,s)|^2ds\right)\right]= 0.
			\end{aligned}
		\end{equation}
		From \eqref{4.9_1}-\eqref{4.9_3}, we obtain
		\begin{equation}\label{3.6}
			\begin{aligned}
				&\hat{\mathbb{E}}\left[\int_{0}^{T}e^{-\beta t}\left|\int_{0}^{t}\tilde{\sigma}(t,s)dB_s\right|^2dt\right]
				\\\leq&\hat{\mathbb{E}}\left[\int_{0}^{T}e^{-\beta t}\left(\left|\int_{0}^{t}\tilde{\sigma}(t,s)dB_s\right|^2-\bar{\sigma}^2\int_{0}^{t}|\tilde{\sigma}(t,s)|^2ds\right)dt\right]+\bar{\sigma}^2\hat{\mathbb{E}}\left[\int_{0}^{T}e^{-\beta t}\int_{0}^{t}|\tilde{\sigma}(t,s)|^2dsdt\right]
				\\\leq&\int_{0}^{T}e^{-\beta t}\hat{\mathbb{E}}\left[\left|\int_{0}^{t}\tilde{\sigma}(t,s)dB_s\right|^2-\bar{\sigma}^2\int_{0}^{t}|\tilde{\sigma}(t,s)|^2ds\right]dt+ \bar{\sigma}^2\hat{\mathbb{E}}\left[\int_{0}^{T}e^{-\beta t}\int_{0}^{t}|\tilde{\sigma}(t,s)|^2dsdt\right]
				\\\leq& \bar{\sigma}^2L^2\hat{\mathbb{E}}\left[\int_{0}^{T}e^{-\beta t}\int_{0}^{t}\left|X(s)-Y(s)\right|^2dsdt\right].
			\end{aligned}
		\end{equation}
		Combining \eqref{3.5} with \eqref{3.6}, we conclude that
		\begin{equation}\label{3.7}
			\begin{aligned}
				\hat{\mathbb{E}}\left[\int_{0}^{T}e^{-\beta t}\left|\Lambda_{t}(X)-\Lambda_{t}(Y)\right|^2dt\right]&\leq 3(T+\bar{\sigma}^4T+\bar{\sigma}^2)L^2\hat{\mathbb{E}}\left[\int_{0}^{T}e^{-\beta t}\int_{0}^{t}\left|X(s)-Y(s)\right|^2dsdt\right]
				\\&= 3(T+\bar{\sigma}^4T+\bar{\sigma}^2)L^2\hat{\mathbb{E}}\left[\int_{0}^{T}\int_{s}^{T}e^{-\beta t}\left|X(s)-Y(s)\right|^2dtds\right]
				\\&\leq\frac{3(T+\bar{\sigma}^4T+\bar{\sigma}^2)L^2}{\beta}\hat{\mathbb{E}}\left[\int_{0}^{T}e^{-\beta s}\left|X(s)-Y(s)\right|^2ds\right].
			\end{aligned}
		\end{equation}
		Set $\beta=6(T+\bar{\sigma}^4T+\bar{\sigma}^2)L^2$, then
		\[\hat{\mathbb{E}}\left[\int_{0}^{T}e^{-\beta t}\left|\Lambda_{t}(X)-\Lambda_{t}(Y)\right|^2dt\right]\leq\frac{1}{2}\hat{\mathbb{E}}\left[\int_{0}^{T}e^{-\beta s}\left|X(s)-Y(s)\right|^2ds\right].
		\]
		In conclusion, $ \Lambda$ is a contraction mapping, which has a unique fixed point $X(\cdot) $. Then equation (\ref{G-SVIE}) has a unique solution $ X(\cdot) \in M_{G}^{2}(0,T) $. 
		
		\textbf{Continuity. }
	If $  \phi(\cdot)\in \tilde{M}_{G}^{2}(0,T)  $ is mean-square continuous, we can obtain from Lemma \ref{sup} that $\sup\limits_{t\in[0,T]}\hat{\mathbb{E}}[|\phi(t)|^{2}]<\infty $. 
	Recalling the assumptions (H2) and \eqref{ito}, we have 
		\[
		\hat{\mathbb{E}}\left[\left| \int_{0}^{t} b(t, s, X(s)) ds\right| ^{2}\right]+\hat{\mathbb{E}}\left[ \left|\int_{0}^{t} h(t, s, X(s)) d\left\langle B\right\rangle_s\right|^{2}\right] \leq (\bar{\sigma}^{4}+1)TL^2\hat{\mathbb{E}}\left[ \int_{0}^{T} 1+| X(s)| ^2ds\right]
		\]
		and
		\[
		\hat{\mathbb{E}}\left[\left|\int_{0}^{t} \sigma(t, s, X(s)) dB_s\right|^{2}\right]\leq\bar{\sigma}^{2}L^2\hat{\mathbb{E}}\left[ \int_{0}^{T} 1+| X(s)| ^2ds\right].
		\]
		Thus, we conclude that
		\begin{equation}\label{jie_contin_1}
			\begin{aligned}
				\sup_{t\in[0,T]}\hat{\mathbb{E}}[|X(t)|^{2}]
				\leq{}&4\sup_{t\in[0,T]}\hat{\mathbb{E}}[|\phi(t)|^{2}]+4L^{2}\left((\bar{\sigma}^{4}+1)T+\bar{\sigma}^{2}\right)\hat{\mathbb{E}}\left[ \int_{0}^{T} 1+| X(s)| ^2ds\right]
				\\<{}&\infty.
			\end{aligned}
		\end{equation}
		Let $  0\leq t\leq t' \leq T$, using \eqref{ito} and the assumptions (H2)-(H3), we have
		\begin{equation}\label{jie_contin_2}
			\begin{aligned}
				{}&\hat{\mathbb{E}}\left[\left|\int_{0}^{t'}\sigma(t',s,X(s))dB_s-\int_{0}^{t}\sigma(t,s,X(s))dB_s\right|^{2}\right]\\
				={}&\hat{\mathbb{E}}\left[\left|\int_{0}^{t}\sigma(t',s,X(s))-\sigma(t,s,X(s))dB_s+\int_{t}^{t'}\sigma(t',s,X(s))dB_s\right|^{2}\right]\\
				\leq{}&2\bar{\sigma}^2\left(\hat{\mathbb{E}}\left[\int_{0}^{t}\rho^{2}(|t'-t|)ds\right] +\hat{\mathbb{E}}\left[\int_{t}^{t'}L^{2}(1+|X(s)|^{2})ds\right]  \right)\\
				\leq{}&2\bar{\sigma}^2T\rho^{2}(|t'-t|)+2\bar{\sigma}^2L^2|t'-t| \sup_{t\in[0,T]}\hat{\mathbb{E}}[1+|X(t)|^{2}].
			\end{aligned}
		\end{equation}
		Sending $  t'\rightarrow t$, it follows that 
		\begin{equation}\label{jie_contin_3}
			\hat{\mathbb{E}}\left[\left|\int_{0}^{t'}\sigma(t',s,X(s))dB_s-\int_{0}^{t}\sigma(t,s,X(s))dB_s\right|^{2}\right]\rightarrow0,	
		\end{equation}
		which implies that the intergral $\int_{0}^{t}\sigma(t,s,X(s))dB_s,\ t\in[0,T]$, is mean-square continuous. Similarly, 
		\[  \int_{0}^{t} b(t, s, X(s)) ds+ \int_{0}^{t} h(t, s, X(s)) d\left\langle B\right\rangle_s ,\quad t\in[0,T],\] is also mean-square continuous. By the mean-square continuity of $ \phi(\cdot)\in \tilde{M}_G^2(0,T)$ and Remark \ref{well-def} , we conclude that $ X(\cdot) \in\tilde{M}_G^2(0,T)$ is mean-square continuous. The proof is complete.
	\end{proof}
	
	Now we present a priori estimate for the solutions of $G$-SVIEs.
	\begin{theorem}
		Suppose that $  \phi_i(\cdot)\in \tilde{M}_{G}^{2}(0,T)  $ is mean-square continuous and $ b_i$, $h_i$, $\sigma_i$ satisfy the assumptions (H1)-(H3) for $i=1,2$. Let $ X_{i}(\cdot)\in \tilde{M}_{G}^{2}(0,T),i=1,2 ,$ be solutions of the following $G$-SVIEs
		\begin{align}\label{guji}
			X_{i}(t)=\phi_{i}(t)+\int_{0}^{t} b_{i}(t, s, X_{i}(s)) ds+\int_{0}^{t} h_{i}(t, s, X_{i}(s)) d\left\langle B\right\rangle_s+\int_{0}^{t} \sigma_{i}(t, s, X_{i}(s)) dB_s, \ t\in[0,T].
		\end{align}
		Then for each $t\in[0,T]$,
		\begin{align*}
			\hat{\mathbb{E}}[|X_{1}(t)-X_{2}(t)|^{2}]\leq{}&f(t)+C_2(\bar{\sigma},L,T)\int_{0}^{t}f(r)e^{C_2(\bar{\sigma},L,T)(t-r)}dr,
		\end{align*}
		where  $C_1(\bar{\sigma},T)$, $C_2(\bar{\sigma},T,L)$ are constants depending on $ \bar{\sigma},T$, and $ \bar{\sigma},\ T,\ L $ respectively and \[f(t)=C_1(\bar{\sigma},T)\left( \hat{\mathbb{E}}[|\hat{\phi}(t)|^{2}]+\hat{\mathbb{E}}\left[ \int_{0}^{T}|\hat{b}(t,s)|^2+\hat{h}(t,s)|^2+|\hat{\sigma}(t,s)|^2ds\right]\right),
		\]$\hat{\phi}(t)=\phi_{1}(t)-\phi_{2}(t) $, $\hat{l}(t,s)=l_{1}(t,s,X_{2}(s))-l_{2}(t,s,X_{2}(s))$, $\text{ }l=b$, $h$, $\sigma$. 
	\end{theorem}	
	
	\begin{proof}[Proof] 
		Recalling (H2), we obtain 
		\begin{equation}\label{guji_1}
			\begin{aligned}
				\left|b_1(t,s,X_1(s))-b_2(t,s,X_2(s))\right| &\leq
				\left|b_1(t,s,X_1(s))-b_1(t,s,X_2(s))\right| +|\hat{b}(t,s)|
				\\&\leq L\left|X_1(s)-X_2(s)\right|+|\hat{b}(t,s)|	.
			\end{aligned}
		\end{equation}
		And the same applies to $ h_1(t,s,X_1(s))-h_2(t,s,X_2(s))$ and $\sigma_1(t,s,X_1(s))-\sigma_2(t,s,X_2(s))$. From \eqref{ito}, we get
		\begin{equation}\label{guji_2}
			\begin{aligned}
				&\hat{\mathbb{E}}\left[ \left| \int_{0}^{t}\sigma_1(t,s,X_1(s))-\sigma_2(t,s,X_2(s))dB_s\right| ^2\right] 
				\\\leq&2\bar{\sigma}^2L^2\hat{\mathbb{E}}\left[ \int_{0}^{t}\left|X_1(s)-X_2(s)\right|^2ds\right] +2\bar{\sigma}^2\hat{\mathbb{E}}\left[ \int_{0}^{t}\left| \hat{\sigma}(t,s)\right| ^2ds\right] .
			\end{aligned}
		\end{equation}
		Set $\hat{X}(t)=X_1(t)-X_2(t) .$ Then by a simple calculation, we derive that for each $ t\in[0,T] $, 
		\begin{align*}
			{}&\hat{\mathbb{E}}[|\hat{X}(t)|^{2}]
			\\\leq{}&C_1(\bar{\sigma},T)\left\lbrace  \hat{\mathbb{E}}[|\hat{\phi}(t)|^{2}]+\hat{\mathbb{E}}\left[ \int_{0}^{T}|\hat{b}(t,s)|^2+\hat{h}(t,s)|^2+|\hat{\sigma}(t,s)|^2ds\right]\right\rbrace+C_2(\bar{\sigma},L,T)\int_{0}^{t}\hat{\mathbb{E}}\left[ |\hat{X}(s)|^{2}\right]ds,
		\end{align*} 
		where $C_1(\bar{\sigma},T)$ is a constant depending on $ \bar{\sigma},T$ and $C_2(\bar{\sigma},T,L)$  is a constant depending on $ \bar{\sigma},\ T,\ L.$  Then, by Gronwall's inequality, we conclude that 
		\begin{align*}
			\hat{\mathbb{E}}[|\hat{X}(t)|^{2}]\leq{}&f(t)+C_2(\bar{\sigma},T,L)\int_{0}^{t}f(r)e^{C_2(\bar{\sigma},L,T)(t-r)}dr,   \quad  t\in[0,T],
		\end{align*}
		where \[f(t)=C_1(\bar{\sigma},T)\left(\hat{\mathbb{E}}[|\hat{\phi}(t)|^{2}]+\hat{\mathbb{E}}\left[ \int_{0}^{T}|\hat{b}(t,s)|^2+\hat{h}(t,s)|^2+|\hat{\sigma}(t,s)|^2ds\right]\right).\]
		The estimate for \eqref{guji} is obtained.
	\end{proof}
	
	To ensure the pathwise continuity of solutions to $G$-SVIEs \eqref{G-SVIE}, we need the following assumptions.
	
	\begin{description}
		\item[(H4)] For all $ x, \ y\in\mathbb{R} $ and $ (t,s)$, $(t',s) \in\Delta $, there exist constants $ C_{T},\ \alpha>0 $ such that
		\[   |\sigma(t',s,x)-\sigma(t,s,x)|\leq C_{T}|t-t'|^{\alpha}.\]
	\end{description}
	
	Next, we establish the existence and uniqueness theorem for solutions with continuous sample  paths.
	\begin{theorem}\label{th3.8}
		Suppose that $  \phi(\cdot)\in M_{G}^{2+\varepsilon}(0,T)  $  for some $ \varepsilon>0$ and $ b$, $ h$, $\sigma $ satisfy the assumptions (H1)-(H3). Then equation (\ref{G-SVIE}) admits a unique solution $ X(\cdot)\in M_{G}^{2+\varepsilon}(0,T) $. Moreover, if $  \phi(\cdot)\in \tilde{M}_{G}^{2+\varepsilon}(0,T)$ has continuous paths with $ \sup\limits_{t\in[0,T]}\hat{\mathbb{E}}[|\phi(t)|^{2+\varepsilon}]<\infty $ and $\sigma$ satisfies the assumptions (H4)  with $ \alpha>\frac{1}{2+\varepsilon} $, then the solution $ X(\cdot) $ of equation (\ref{G-SVIE}) has a continuous modification.
	\end{theorem}
	
	\begin{proof}[Proof]
		The proof of the existence and uniqueness of the solution is similar to  Theorem \ref{jie1}. Thus, equation (\ref{G-SVIE}) has a unique solution  $ X(\cdot) \in M_{G}^{2+\varepsilon}(0,T) $.   Moreover, since $  \phi(\cdot)\in \tilde{M}_{G}^{2+\varepsilon}(0,T)$, we get $ X(\cdot)\in \tilde{M}_{G}^{2+\varepsilon}(0,T) $. Then, following the proof of inequality \eqref{jie_contin_1},  we derive that $\sup_{t\in[0,T]}\hat{\mathbb{E}}[|X(t)|^{2+\varepsilon}]<\infty$. Next, we consider the pathwise continuity of solution (\ref{G-SVIE}). 
		
		On the one hand, for each $ t'\geq t $, we have
		\begin{align*}
			&\left| \int_{0}^{t'}\sigma(t',s,X(s))dB_s-\int_{0}^{t}\sigma(t,s,X(s))dB_s\right| 
			\\\leq&\left|\int_{0}^{t}\sigma(t’,s,X(s))-\sigma(t,s,X(s))dB_s\right| +\left| \int_{t}^{t'}\sigma(t’,s,X(s))dB_s\right| .
		\end{align*}
		By BDG's inequality and H\"{o}lder's inequality, we obtain from the assumption (H2), (H4) that
		\begin{align*}
			{}&\hat{\mathbb{E}}\left[ \left| \int_{0}^{t'}\sigma(t',s,X(s))dB_s-\int_{0}^{t}\sigma(t,s,X(s))dB_s\right| ^{2+\varepsilon}\right]\\
			\leq{}&C(\varepsilon)\left(\hat{\mathbb{E}}\left[ \left| \int_{0}^{t}\sigma(t',s,X(s))-\sigma(t,s,X(s))dB_s\right| ^{2+\varepsilon}\right] +\hat{\mathbb{E}}\left[ \left|\int_{t}^{t'}\sigma(t',s,X(s))dB_s\right| ^{2+\varepsilon}\right] \right)\\
			\leq{}&C(\bar{\sigma},\varepsilon)\left(\hat{\mathbb{E}}\left[ \left| \int_{0}^{t}C_{T}|t'-t|^{2\alpha}ds\right| ^{\frac{2+\varepsilon}{2}}\right] +\hat{\mathbb{E}}\left[ \left| \int_{t}^{t'}L^{2}(1+|X(s)|^{2})ds\right| ^{\frac{2+\varepsilon}{2}}\right] \right)\\
			\leq{}&C(\bar{\sigma},T,L,\varepsilon)\left(|t'-t|^{(2+\varepsilon)\alpha}+|t'-t|^{1+\frac{\varepsilon}{2}} \sup_{t\in[0,T]}\hat{\mathbb{E}}[1+|X(t)|^{2+\varepsilon}]\right),
		\end{align*}
		where the constant $C(\bar{\sigma},T,L,\varepsilon)$ depends on $\bar{\sigma},T,L,\varepsilon. $
		Note that both $ (2+\varepsilon)\alpha $ and $1+\frac{\varepsilon}{2}  $ are greater than $ 1 $. Then, using Theorem \ref{continuous}, we conclude that there exists a $ \delta $-order H\"{o}lder continuous modification of $ \int_{0}^{t}\sigma(t,s,X(s))dB_s,\ t\in[0,T] $, for $ \delta\in(0,\frac{(2+\varepsilon)\alpha-1}{2+\varepsilon}\wedge \frac{\varepsilon}{4+2\varepsilon}) $. 
		
		On the other hand, under the assumptions (H2)-(H3), we have that for fixed $\omega\in \Omega$, 
		\begin{equation}\label{ds_contin}
			\begin{aligned}
				\left| \int_{0}^{t'} b(t', s, X(s)) ds-\int_{0}^{t} b(t, s, X(s)) ds\right| &\leq \left| \int_{0}^{t} b(t', s, X(s))-b(t, s, X(s))ds\right| +\left| \int_{t}^{t'} b(t', s, X(s))ds\right| \\&\leq\rho(|t'-t|)T+2L^2\left(\int_{0}^{T} (1+|X(s)|^2)ds\right)^{\frac{1}{2}}|t'-t|^{\frac{1}{2}}.
			\end{aligned}
		\end{equation}
		Since $ X(\cdot) \in M_{G}^{2+\varepsilon}(0,T) $, we obtain $\left(\int_{0}^{T} (1+|X(s)|^2)ds\right)^{\frac{1}{2}}<\infty \text{ }q.s.$ Letting $ t'\rightarrow t$, we obtain the intergral $\int_{0}^{t} b(t, s, X(s)) ds$, $t\in[0,T]$ has continuous paths q.s. Similarly, we can prove the pathwise continuity of $\int_{0}^{t} h(t, s, X(s)) d\left\langle B\right\rangle_s.$
		Thus, based on the pathwise continuity for $\phi(t)$ in $t$, we establish the desired results. 	\end{proof}	
	\begin{remark}The above results and methods still hold for multi-dimensional $X$ and $G$-Brownian motion $B$.
	\end{remark}

	\section{Comparison theorem of  $ G $-SVIEs}
	\noindent
	
	In this section, we consider the following form of $ G $-SVIEs: for each $t\in[0,T]$,
	\begin{equation}\label{4.1} 
		X_i(t)=\phi_i(t)+\int_{0}^{t}b_i(t,s,X_i(s))ds+\int_{0}^{t}h_i(t,s,X_i(s))d\langle B\rangle_s+H(t)\int_{0}^{t}\sigma(s,X_i(s))dB_s,\quad i=1,2,
	\end{equation}
	where the coefficients  
	\begin{equation*}
		b_i(t, s,\omega, x),\text{ }h_i(t,s,\omega,x):\Delta  \times \Omega \times \mathbb{R}
		\rightarrow \mathbb{R},  
	\end{equation*}
	\[
	\sigma(s,\omega,x):[0,T] \times \Omega \times \mathbb{R}
	\rightarrow \mathbb{R},
	\]
	\[
	H(t,\omega):[0,T] \times \Omega
	\rightarrow \mathbb{R}.
	\] 
	
	In order to obtain the comparison results for the $G-$SVIEs \eqref{4.1}, we need the following additional assumptions. 
	
	\begin{description}
		\item[(A1)]The function $H(\cdot)\in \tilde{M}_G^2(0,T)$ has continuous paths, and there exist positive constants  $m$, $M$ such that
		\[m\leq H(t,\omega)\leq M,\quad t\in[0,T],\ \omega\in\Omega;	\] 
		\item[(A2)]
		For each $\omega\in \Omega$, $x\in\mathbb{R}$, $t'\geq t\geq s $, 
		\[
		\frac{b_1(t',s,x)-b_2(t',s,x)}{H(t')}\geq \frac{b_1(t,s,x)-b_2(t,s,x)}{H(t)}
		\geq0,
		\]
		\[
		\frac{h_1(t',s,x)-h_2(t',s,x)}{H(t')}\geq \frac{h_1(t,s,x)-h_2(t,s,x)}{H(t)}
		\geq0;
		\]
		\item[(A3)]There exists $i=1$ or $2$ such that  for each $\omega\in \Omega$, $x,y\in\mathbb{R}$, $ y\geq x$ and $t'\geq t\geq s $, 
		\[
		\frac{l_i(t',s,y)-l_i(t',s,x)}{H(t')}\geq \frac{l_i(t,s,y)-l_i(t,s,x)}{H(t)},
		\]
		where $ l_i=b_i,h_i$ and $i=1,2$;
		\item[(A4)] For each $\omega\in \Omega$, $  t'\geq t $, 
		\[\frac{\phi_1(t')-\phi_2(t')}{H(t')}\geq\frac{\phi_1(t)-\phi_2(t)}{H(t)}\geq0.\]
	\end{description}
	\begin{remark}\label{dandiao}
		The assumption (A3) holds if and only if there exists $i=1$ or $2$ such that for any $ y\geq x$ and $t'\geq t $,
		\[
		\frac{l_i(t',s,y)}{H(t')}-\frac{l_i(t,s,y)}{H(t)}\geq\frac{l_i(t',s,x)}{H(t')}-\frac{l_i(t,s,x)}{H(t)},
		\]
		where $ l=b_i,h_i$ and $i=1,2$. This inequality implies that the function $\frac{l_i(t',s,y)}{H(t')}-\frac{l_i(t,s,y)}{H(t)}$ is a non-decreasing function in $y$. 
	\end{remark}
	
	Using the main results in Theorem \ref{jie1}, we first establish the existence and uniqueness of solutions to $G$-SVIEs \eqref{4.1} as follows.
	\begin{proposition}\label{4jie_noncontin}
		Suppose that $  \phi_i(\cdot)\in \tilde{M}_{G}^{2}(0,T)  $ satisfy $ \sup\limits_{t\in[0,T]}\hat{\mathbb{E}}[|\phi_i(t)|^{2}]<\infty$, $H$ satisfies (A1) and $ b_i$, $ h_i$, $\sigma$ satisfy (H1)-(H3)  for $i=1,2$.  Then $G-$SVIEs \eqref{4.1} admit a unique solution $ X_i(\cdot)\in \tilde{M}_G^2(0,T) $ for $i=1,2 $ with $ \sup\limits_{t\in[0,T]}\hat{\mathbb{E}}[|X_i(t)|^{2}]<\infty$.
	\end{proposition}
	\begin{proof}
		Note that the coefficient $\sigma$ is independent of $t$, which satisfies the assumption (H3) directly. Since $b$, $h$, $\sigma$ satisfy (H1)-(H3) and $H$ satisfies (A1),  we obtain from Lemma \eqref{MG} that the integral  
		\begin{equation}\label{jifen}
		\int_{0}^{t}b_i(t,s,X_i(s))ds+\int_{0}^{t}h_i(t,s,X_i(s))d\langle B\rangle_s+H(t)\int_{0}^{t}\sigma(s,X_i(s))dB_s\in \tilde{M}_{G}^{2}(0,T), \quad t\in[0,T],\text{ }i=1,2,
	\end{equation}which is well-defined. And the existence and uniqueness of the solution to equation \eqref{4.1} can be derived from Theorem \ref{jie1} directly. Following the proof of \eqref{jie_contin_1}, we have $ \sup\limits_{t\in[0,T]}\hat{\mathbb{E}}[|X_i(t)|^{2}]<\infty$ for $i=1,2$.
	\end{proof}
	\begin{remark}\label{shuoming}
		(1) Since $H(t)$ does not satisfy the assumption (H3), the integral $H(t)\int_{0}^{t}\sigma(s,X_i(s))dB_s$, $t\in[0,T]$ is not mean-square for $i=1,2$. Thus, $X_i(\cdot)$ is not mean-square continuous under the condition of Proposition \ref{4jie_noncontin} for $i=1,2$.
		\\(2) From the proof of Lemma \ref{MG}, we know that the effect of (H3) is to ensure that the integral \eqref{jifen} belongs to $\tilde{M}_{G}^{2}(0,T)$. If the coefficients $b_i$, $h_i$, $\sigma$, $i=1,2$ satisfy (H1)-(H2)  and the integral \eqref{jifen} is well-defined in the following, the conclusion of proposition \ref{4jie_noncontin} also holds.
	\end{remark}
	
	\begin{proposition}\label{4jie_contin}
		Suppose that $  \phi_i(\cdot)\in \tilde{M}_{G}^{2}(0,T)  $ has continuous paths with $ \sup\limits_{t\in[0,T]}\hat{\mathbb{E}}[|\phi_i(t)|^{2}]<\infty$, $H$ satisfies (A1) and $ b_i$, $h_i$, $\sigma$ satisfy (H1)-(H3)  for $i=1,2$.  Then $G$-SVIEs \eqref{4.1} admit a unique  solution $ X_i(\cdot)\in \tilde{M}_G^2(0,T) $ satisfying $ \sup\limits_{t\in[0,T]}\hat{\mathbb{E}}[|X_i(t)|^{2}]<\infty$ for $i=1,2 $. Moreover, $ X_i(\cdot)$ has continuous paths for $i=1,2 $.
	\end{proposition}
	
	\begin{proof}
		Since $H(\cdot)$ has continuous paths,  it is clear that $H(t)\int_{0}^{t}\sigma(s,X_i(s))dB_s, t\in[0,T]$ has continuous paths for $i=1,2$. Following the proof of inequality \eqref{ds_contin}, it is easy to verify that the integral $\int_{0}^{t} b_i(t, s, X_i(s)) ds+\int_{0}^{t} h_i(t, s, X_i(s)) d\left\langle B\right\rangle_s, t\in[0,T]$ also has continuous paths for $i=1,2$. Then we can derive that $X_i(\cdot)$ has continuous sample paths for $i=1,2.$ The existence and uniqueness of the solution to \eqref{4.1} can be derived from Theorem \ref{jie1} similarly.
	\end{proof}
	
	In addition, we emphasize that the comparison theorem for $G$-SVIEs with  general Volterra-type diffusion coefficients may fail, even when $G $ is a linear function (see Remark 2 in \cite{tudor1989a}). So we consider the $ G $-SVIEs with separable Volterra-type diffusion coefficient, where  the coefficient  $\sigma(t,s,x)$ takes the form $ H(t)\sigma(s,x) $ (see \cite{ferreyra2000comparison}).  
	
	Set $ \bar{X}_i(t)=\frac{X_i(t)}{H(t)} , i=1,2$, $ \hat{X}(t) =\bar{X}_1(t)-\bar{X}_2(t)$ and $ \hat{\phi}(t)=\phi_1(t)-\phi_2(t) $. From \eqref{4.1} we obtain that for $i=1,2$,
	\begin{equation}\label{2}
		\begin{aligned}
			\bar{X}_i(t) &= \frac{\phi_i(t)}{H(t)}+\int_{0}^{t}\frac{b_{i}(t,s,\bar{X}_i(s)H(s))}{H(t)}ds+\int_{0}^{t}\frac{h_{i}(t,s,\bar{X}_i(s)H(s))}{H(t)}d\langle B\rangle_s+\int_{0}^{t}\sigma(s,\bar{X}_i(s)H(s))dB_s.
		\end{aligned}
	\end{equation}
	Then 
	\begin{equation}\label{3}
		\begin{aligned}
			\hat{X}(t) = \frac{\hat{\phi}(t)}{H(t)}
			+\int_{0}^{t}\frac{\tilde{b}(t,s)}{H(t)}ds
			+\int_{0}^{t}\frac{\tilde{h}(t,s)}{H(t)}d\langle B\rangle_s+\int_{0}^{t}\tilde{\sigma}(s)dB_s,
		\end{aligned}
	\end{equation}
	where $ \tilde{b}(t,s)=b_{1}(t,s,\bar{X}_1(s)H(s))-b_{2}(t,s,\bar{X}_2(s)H(s)) $,
	$ \tilde{h}(t,s)= h_{1}(t,s,\bar{X}_1(s)H(s))-h_{2}(t,s,\bar{X}_2(s)H(s)) $,
	$ \tilde{\sigma}(s)=\sigma(s,\bar{X}_1(s)H(s))-\sigma(s,\bar{X}_2(s)H(s)). $ 
	
	Denote the Lipschitz function sequences 
	$\gamma_n(y):=I_{\left\lbrace |y|\leq\frac{1}{n}\right\rbrace}+(2-n|y|)I_{\left\lbrace \frac{1}{n}\leq|y|\leq\frac{2}{n}\right\rbrace}$, $ \forall n\in \mathbb{N} $. And set 
	\[	(1-\gamma_n(y))y^{-1} =
	\begin{cases}
		\frac{1}{y},&  |y|\geq\frac{2}{n},
		\\(n|y|-1)y^{-1},&  \frac{1}{n}\leq|y|\leq\frac{2}{n},
		\\0,&|y|\leq\frac{1}{n}.
	\end{cases}
	\] Based on this and the quasilinearization for $ b_2,\ h_2,\ \sigma $, we firstly construct the following sequences of $G$-SVIEs, i.e.
	\begin{equation}\label{8}
		\begin{aligned}
			\hat{X}^{n}(t) =&\frac{\hat{\phi}(t)}{H(t)}+\int_{0}^{t}\frac{b^{n}(t,s)\hat{X}^{n}(s)+\hat{b}(t,s)}{H(t)}ds+\int_{0}^{t}\frac{h^{n}(t,s)\hat{X}^{n}(s)+\hat{h}(t,s)}{H(t)}d\langle B\rangle_s
			\\&+\int_{0}^{t}\sigma^n(s)\hat{X}^{n}(s)dB_s,\quad n\in\mathbb{N},
		\end{aligned}
	\end{equation}
	where 	
	\begin{equation}\label{def_n}
		\begin{aligned}
			&l^{n}(t,s)=\frac{1-\gamma_n(\hat{X}(s))}{\hat{X}(s)}(l_2(t,s,\bar{X}_1(s)H(s))-l_2(t,s,\bar{X}_2(s)H(s))),\quad l=b,h,
			\\&\hat{l}(t,s)=l_1(t,s,\bar{X}_1(s)H(s))-l_2(t,s,\bar{X}_1(s)H(s)),\quad l=b,h,
			\\&\sigma^{n}(s)=\frac{1-\gamma_n(\hat{X}(s))}{\hat{X}(s)}(\sigma(s,\bar{X}_1(s)H(s))-\sigma(s,\bar{X}_2(s)H(s))).		
		\end{aligned}
	\end{equation}
	
	Now, we will prove the convergence result of $\hat{X}^{n}$.
	\begin{lemma}\label{nixianxing}
		Suppose that $  \phi_i(\cdot)\in \tilde{M}_{G}^{2}(0,T)  $ satisfy $ \sup\limits_{t\in[0,T]}\hat{\mathbb{E}}[|\phi_i(t)|^{2}]<\infty$, $H$ satisfies (A1) and $ b_i$, $h_i$, $ \sigma$ satisfy the assumptions (H1)-(H3)  for $i=1,2$. Let $ X_i(\cdot),\ i=1,2 $ be the solutions of $G-$SVIEs \eqref{4.1} corresponding to $ \phi_i$, $b_i$, $h_i$, $\sigma$ and $H$. Then
		\begin{equation}\label{X_n_X}
			\lim_{n\rightarrow \infty}\hat{\mathbb{E}}\left[\sup\limits_{t\leq T}|\hat{X}(t)-\hat{X}^n(t)|^2\right]=0.
		\end{equation}
	\end{lemma}
	\begin{proof}
		According to the definition of $\gamma_n(\cdot)$, it is easy to check that $  I_{\left[-\frac{1}{n},\frac{1}{n}\right]} \leq \gamma_n(\cdot) \leq I_{\left[-\frac{2}{n},\frac{2}{n}\right]} $. Note that $ (1-\gamma_n(y))y^{-1} $ is continuous function in $y$ and $ |(1-\gamma_n(y))y^{-1}|\leq n$ for each  $n\in\mathbb{N}$. Recalling (H1)-(H2), we can deduce that $ b^{n}(t,s),h^{n}(t,s), \sigma^{n}(s) $ belong to $ M_G^2(0,T)$ by \cite{hu2016quasi} for each $t\in[0,T]$ and $|b^{n}(t,s)|\leq 2LH(s),|h^{n}(t,s)|\leq 2LH(s), |\sigma^{n}(s)|\leq 2LH(s)$. Note that the integral
		\begin{equation}\label{yuxiang_1}
			\int_{0}^{t}\hat{b}(t,s)ds+\int_{0}^{t}\hat{h}(t,s)d\langle B\rangle_s
		\end{equation}
		belongs to $\tilde{M}_{G}^{2}(0,T) $, which satisfies that 
		\begin{equation}\label{yuxiang_2}
			\sup\limits_{t\leq T}\hat{\mathbb{E}}\left[ \left| \int_{0}^{t}\hat{b}(t,s)ds\right|^2+\left| \int_{0}^{t}\hat{h}(t,s)d\langle B\rangle_s\right| ^2\right] \leq 4(1+\bar{\sigma}^4)TL^2\hat{\mathbb{E}}\left[\int_{0}^{T}(1+M^2|\bar{X}_1(s)|^2)ds\right] <\infty.
		\end{equation}
		Thus, we can replace $\frac{\hat{\phi}(t)}{H(t)}$ by $\frac{\hat{\phi}(t)}{H(t)}+\int_{0}^{t}\hat{b}(t,s)ds+\int_{0}^{t}\hat{h}(t,s)d\langle B\rangle_s$ for each $t\in[0,T]$.
		It is easy to verify that the coefficients of equation \eqref{8} satisfy (H1)-(H2).  For each $(t,s),\ (t',s)\in \Delta$, $n\in\mathbb{N}$,
		\begin{equation}\label{b_n_mo_contin}
			\begin{aligned}
				|b^n(t',s)-b^n(t,s)|\leq&\left|\frac{1-\gamma_n(\hat{X}(s))}{\hat{X}(s)}\right||b_2(t',s,\bar{X}_1(s)H(s))-b_2(t,s,\bar{X}_1(s)H(s))|
				\\&+\left|\frac{1-\gamma_n(\hat{X}(s))}{\hat{X}(s)}\right||b_2(t,s,\bar{X}_2(s)H(s))-b_2(t',s,\bar{X}_2(s)H(s))|
				\\\leq&2n\rho(|t'-t|)
			\end{aligned}
		\end{equation}
		and $|h^n(t',s)-h^n(t,s)|\leq2n\rho(|t'-t|). $ Similarly, we have $|\hat{b}(t',s)-\hat{b}(t,s)|\leq 2\rho(|t'-t|)$ and $|\hat{h}(t',s)-\hat{h}(t,s)|\leq 2\rho(|t'-t|)$. Then, by Lemma \ref{MG}, we get the integral $
		\int_{0}^{\cdot} (b^{n}(\cdot,s)Z(s)+\hat{b}(\cdot,s)) ds+ \int_{0}^{\cdot} (h^{n}(\cdot,s)Z(s)+\hat{h}(\cdot,s))d\left\langle B\right\rangle_s \in \tilde{M}_G^2(0,T)$ for each $ Z(\cdot)\in M_G^2(0,T)$. Since $H(\cdot)$ is bounded, we can obtain that for each $ Z(\cdot)\in M_G^2(0,T)$, \[H(\cdot)^{-1}\left( 
		\int_{0}^{\cdot} (b^{n}(\cdot,s)Z(s)+\hat{b}(\cdot,s) )ds+ \int_{0}^{\cdot} (h^{n}(\cdot,s)Z(s)+\hat{h}(\cdot,s))d\left\langle B\right\rangle_s\right) \in \tilde{M}_G^2(0,T),\] which is well-defined.	Combining Proposition \ref{4jie_noncontin} with Remark \ref{shuoming} (2), the equations $G$-SVIEs \eqref{8} admit a unique solution $ \hat{X}^n(\cdot)\in \tilde{M}_G^2(0,T) $ for $n\in \mathbb{N}$.  Following the proof of \eqref{jie_contin_1}, we have 
		\begin{equation}\label{X_n<infty}
		\sup\limits_{t\in[0,T]}\hat{\mathbb{E}}\left[|\hat{X}^n(t)|^2\right]<\infty.
		\end{equation} 
		
		Applying linearization techniques to \eqref{3}, it yields that
		\begin{equation}\label{4}
			\begin{aligned}
				\hat{X}(t) =&\frac{\hat{\phi}(t)}{H(t)}+\int_{0}^{t}\frac{b^{n}(t,s)\hat{X}(s)+\hat{b}(t,s)+a^n(t,s)}{H(t)}ds
				\\&
				+\int_{0}^{t}\frac{h^{n}(t,s)\hat{X}(s)+\hat{h}(t,s)+c^n(t,s)}{H(t)}d\langle B\rangle_s+\int_{0}^{t}\left(\sigma^n(s)\hat{X}(s)+d^n(s)\right)dB_s,
			\end{aligned}
		\end{equation}
		where $a^{n}(t,s)=\gamma_n(\hat{X}(s))(b_2(t,s,\bar{X}_1(s)H(s))-b_2(t,s,\bar{X}_2(s)H(s))) $, $c^{n}(t,s)=\gamma_n(\hat{X}(s))(h_2(t,s,\bar{X}_1(s)H(s))-h_2(t,s,\bar{X}_2(s)H(s))) $, $d^{n}(s)=\gamma_n(\hat{X}(s))(\sigma(s,\bar{X}_1(s)H(s))-\sigma(s,\bar{X}_2(s)H(s)))$ and $b^{n}(t,s),h^{n}(t,s),\hat{b}(t,s),\hat{h}(t,s)$, $\sigma^{n}(s)  $	are defined in \eqref{def_n}. It is easy to verify that \[
		|a^{n}(t,s)|\leq L|\hat{X}(s)|H(s)\gamma_n(\hat{X}(s))\leq \frac{2}{n}LH(s).
		\]
		Similarly, we derive $ |c^{n}(t,s)|\leq \frac{2}{n}LH(s),  |d^{n}(s)|\leq \frac{2}{n}LH(s)$.  From \eqref{8} and \eqref{4}, we get
		\begin{equation}\label{4.13_1}
			\begin{aligned}
				\hat{X}(t)-\hat{X}^{n}(t)&= \int_{0}^{t}\frac{b^{n}(t,s)(\hat{X}(s)-\hat{X}^{n}(s))+a^n(t,s)}{H(t)}ds
				+\int_{0}^{t}\frac{h^{n}(t,s)(\hat{X}(s)-\hat{X}^{n}(s))+c^n(t,s)}{H(t)}d\langle B\rangle_s\\&
				+\int_{0}^{t}\left(\sigma^n(s)(\hat{X}(s)-\hat{X}^{n}(s))+d^n(s)\right)dB_s.
			\end{aligned}
		\end{equation}
		By BDG's inequality \eqref{bdg}, we derive for each $u\in[0,T],$
		\begin{equation}\label{13}
			\begin{aligned}
				&\hat{\mathbb{E}}\left[\sup\limits_{t\leq u}|\hat{X}(t)-\hat{X}^n(t)|^2\right]
				\\\leq&6\left\{2(1+\bar{\sigma}^4)T\hat{\mathbb{E}}\left[\sup\limits_{t\leq u}\int_{0}^{u}\frac{\left(|b^n(t,s)|^2+|h^n(t,s)|^2\right)|\hat{X}(s)-\hat{X}^n(s)|^2}{H^2(t)}ds\right]\right.
				\\&\left.+2(1+\bar{\sigma}^4)T\hat{\mathbb{E}}\left[\sup\limits_{t\leq u}\int_{0}^{u}(\frac{|a^n(t,s)|^2+|c^n(t,s)|^2}{H^2(t)}+|d^n(s)|^2)ds\right]+\hat{\mathbb{E}}\left[\sup\limits_{t\leq u}\left|\int_{0}^{t}\sigma^n(s)(\hat{X}(s)-\hat{X}^n(s))dB_s\right|^2\right]
				\right\}
				\\\leq& C(\bar{\sigma},L,T)\left(\frac{M^2}{m^2}+M^2\right)\left\{\frac{1}{n^2}+\hat{\mathbb{E}}\left[\int_{0}^{u}|\hat{X}(s)-\hat{X}^n(s)|^2ds\right]\right\},
			\end{aligned}
		\end{equation}
		where $C(\bar{\sigma},L,T)$ depends on $\bar{\sigma},L,T$. Set $u=T$, we have $\hat{\mathbb{E}}\left[\sup\limits_{t\leq T}|\hat{X}(t)-\hat{X}^n(t)|^2\right]<\infty$.  Then, by \eqref{13}, we obtain 
\begin{equation}\label{gronwall_1}
\hat{\mathbb{E}}\left[\sup\limits_{t\leq u}|\hat{X}(t)-\hat{X}^n(t)|^2\right]\leq C(\bar{\sigma},L,T)\left(\frac{M^2}{m^2}+M^2\right)\left\{\frac{1}{n^2}+\hat{\mathbb{E}}\left[\int_{0}^{u}\sup\limits_{r\leq s}|\hat{X}(r)-\hat{X}^n(r)|^2ds\right]\right\}.
\end{equation}
Applying Gronwall's inequality to \eqref{gronwall_1}, we have
		\[
		\hat{\mathbb{E}}\left[\sup\limits_{t\leq T}|\hat{X}(t)-\hat{X}^n(t)|^2\right]\leq \frac{C}{n^2}e^{CT},\quad n\in\mathbb{N},
		\]
		where $C=C(\bar{\sigma},L,T)\left(\frac{M^2}{m^2}+M^2\right).$ Sending $n\rightarrow \infty$, we conclude that \eqref{X_n_X}.
	\end{proof}
	\begin{remark}
		Note that the new quasilinearization technique for $ b_2,\ h_2,\ \sigma $ ( see \eqref{4}) is different from the classical one, which takes the form of 
		\begin{equation}\label{classical_linear}
			\tilde{b}(t,s)=\hat{b}(t,s)+b(t,s)H(s)(\bar{X}_1(s)-\bar{X}_2(s)),
		\end{equation}
		where $ b(t,s)=\int_{0}^{1}\partial_x b_2\left(t,s,\bar{X}_2(s)H(s)+\alpha(\bar{X}_1(s)-\bar{X}_2(s))H(s)\right)d\alpha$. 
	\end{remark}
	Now we establish the comparison theorem for $G$-SVIEs \eqref{4.1} under the condition that $\phi_i(\cdot)$ has continuous paths for $i=1,2$.
	\begin{theorem}\label{th4.2}
		Suppose that $  \phi_i(\cdot)\in \tilde{M}_{G}^{2}(0,T)  $ has continuous paths with $ \sup\limits_{t\in[0,T]}\hat{\mathbb{E}}[|\phi_i(t)|^{2}]<\infty$, $H$ satisfies (A1) and $ b_i, h_i, \sigma$ satisfy the assumptions (H1)-(H3), (A2)-(A4)  for $i=1,2$. Let $ X_i(\cdot),\ i=1,2 $ be the solutions of $G-$SVIEs \eqref{4.1} corresponding to $ \phi_i,b_i,h_i,\sigma$ and $H$. Then \[X_1(t)\geq X_2(t),\quad  \forall t\in [0,T],\text{ } q.s.
		\]  
	\end{theorem}
	\begin{proof}	
		By (A1) and equation \eqref{3}, we only need to prove \[ \hat{X}(t) \geq 0,  \quad \forall t\in[0,T],\text{ } q.s. \] 
		
		For each fixed $P\in \mathcal{P}$, since $\phi_{i}(t),\ H(t)$ has continuous paths, we define the sequence of stopping times $\{\tau_k^{\delta}\}_{k\in \mathbb{N}}$ as follows, where $\mathcal{P}$ is as given in Lemma \ref{thm2.1}. For any $k\in \mathbb{N},\ \delta>0$, set $\tau_{0}^{\delta} =0$ and
		\[ 
		\tau_{k}^{\delta}=\inf\{t\in (\tau_{k-1}^{\delta}, T]: |\frac{1}{H(t)}-\frac{1}{H(\tau_{k-1}^{\delta})}|\ge\delta \text{ or } |\hat{\phi}(t)-\hat{\phi}(\tau_{k-1}^{\delta})|\ge\delta \} \wedge[(\tau_{k-1}^{\delta}+\delta)\wedge T],
		\] 
		where $\inf \emptyset = \infty$. Then the sequence of stopping times $\{\tau_k^{\delta}\}_{k\in \mathbb{N}}$ satisfies $ 0=\tau_0^\delta<\tau_1^\delta<\tau_2^\delta<\cdots\leq T $ and $\max\limits_{ k\geq 0}|\tau_{k}^{\delta}-\tau_{k-1}^{\delta}| \leq\delta$. For each $\omega\in\Omega$, there exists 
		$m\in \mathbb{N}$ such that $\tau_{m}^{\delta}(\omega)=T.$ Recalling \eqref{8}, we define  for each $(t,s)\in \Delta$,
		\[
		\hat{\phi}^{\delta}(t)=\sum\limits_{k\geq0}\hat{\phi}(\tau_{k}^{\delta})I_{[\tau_{k}^{\delta},\tau_{k+1}^{\delta})}(t),
		\quad  \frac{1}{H^{\delta}(t)}=\sum\limits_{k\geq0}\frac{1}{H(\tau_{k}^{\delta})}I_{[\tau_{k}^{\delta},\tau_{k+1}^{\delta})}(t),
		\]
		\[
		b^{n,\delta}(t,s)=\sum\limits_{k\geq0}b^n(\tau_{k}^{\delta}\vee s,s)I_{[\tau_{k}^{\delta},\tau_{k+1}^{\delta})}(t),
		\quad h^{n,\delta}(t,s)=\sum\limits_{k\geq0}h^n(\tau_{k}^{\delta}\vee s,s)I_{[\tau_{k}^{\delta},\tau_{k+1}^{\delta})}(t)
		\]
		and 
		\[
		\hat{b}^\delta(t,s)=\sum\limits_{k\geq0}\hat{b}(\tau_{k}^{\delta}\vee s,s)I_{[\tau_{k}^{\delta},\tau_{k+1}^{\delta})}(t),
		\quad
		\hat{h}^\delta(t,s)=\sum\limits_{k\geq0}\hat{h}(\tau_{k}^{\delta}\vee s,s)I_{[\tau_{k}^{\delta},\tau_{k+1}^{\delta})}(t).
		\]
		Since the sequence of stopping times $\{\tau_k^{\delta}\}_{k\in \mathbb{N}}$ is constructed under each $P\in \mathcal{P}$, we derive from Lemma \ref{thm2.1} that $b^{n,\delta}(t,s)$, $h^{n,\delta}(t,s)$, $\hat{b}^\delta(t,s) $, $\hat{h}^\delta(t,s)$, $t\in[0,T]$ belong to $M_G^2(0,T)$ for each $\bar{X}_1(\cdot)$, $\bar{X}_2(\cdot) \in M_G^2(0,T)$.
		
		Following this, we construct the approximation sequences of $\{\hat{X}^{n,\delta}_P\}_{n\in\mathbb{N}}$ under $ P $ and verify its convergence properties. For each fixed $ n\in \mathbb{N}, \ t\in [0,T] $, $\delta>0$, 
		\begin{equation}\label{9}
			\begin{aligned}
				\hat{X}^{n,\delta}_P(t) =&\frac{\hat{\phi}^{\delta}(t)}{H^{\delta}(t)}+\int_{0}^{t}\sigma^n(s)\hat{X}^{n,\delta}_P(s)dB_s+\int_{0}^{t}\frac{b^{n,\delta}(t,s)\hat{X}^{n,\delta}_P(s)+\hat{b}^{\delta}(t,s)}{H^{\delta}(t)}ds
				\\&+\int_{0}^{t}\frac{h^{n,\delta}(t,s)\hat{X}^{n,\delta}_P(s)+\hat{h}^{\delta}(t,s)}{H^{\delta}(t)}d\langle B\rangle_s.
			\end{aligned}
		\end{equation}
		And note that $|b^{n,\delta}(t,s)|\leq2LH(s), |h^{n,\delta}(t,s)|\leq2LH(s),|\sigma^n(s)| \leq 2LH(s)$, $\forall n\in \mathbb{N}$. Following the proof of inequalities \eqref{yuxiang_1}-\eqref{yuxiang_2}, we can derive that the coefficients of equation \eqref{9} satisfy assumptions (H1)-(H2). Then by Theorem 3.A. in \cite{berger1980volterra}, the SVIE \eqref{9} under $P$ admits a unique solution $ \hat{X}^{n,\delta}_P(\cdot)\in M_P^2(0,T) $ , where\[
		M_P^2(0,T):=\left\lbrace X:[0,T]\times\Omega:X(\cdot) \text{ is } \mathcal{F}_{t}\text{-adapted, measurable process s.t. } E_{P}\left[\int_{0}^{T}|X(t)|^{2}dt\right]<\infty  \right\rbrace  .
		\] Then, following the proof of \eqref{jie_contin_1}, we have 
		\begin{equation}\label{X_ndelta<infty}
			\sup\limits_{t\in[0,T]}E_P\left[|\hat{X}^{n,\delta}_P(t)|^2\right]<\infty.
		\end{equation} 
		
		Next, for each fixed $ n\in\mathbb{N} $ and $ P\in\mathcal{P} $, we establish the convergence of  $\hat{X}^{n,\delta}_P(t) $ and $\hat{X}^{n}(t)$. Combining \eqref{8} with \eqref{9}, we obtain that 
		\begin{equation}\label{10}
			\begin{aligned}
				&|\hat{X}^{n,\delta}_P(t)-\hat{X}^n(t)|^2
				\\\leq&6\left\{\left|\frac{\hat{\phi}^{\delta}(t)}{H^{\delta}(t)}-\frac{\hat{\phi}(t)}{H(t)}\right|^2+ \left|\int_{0}^{t}\sigma^{n}(s)(\hat{X}^{n,\delta}_P(s)-\hat{X}^n(s))dB_s\right|^2\right.
				\\&\left.+ T\int_{0}^{t}\left|\frac{b^{n,\delta}(t,s)\hat{X}^{n,\delta}_P(s)}{H^{\delta}(t)}-\frac{b^{n}(t,s)\hat{X}^{n}(s)}{H(t)}\right|^2ds+T\int_{0}^{T}\left|\frac{\hat{b}^{\delta}(t,s)}{H^{\delta}(t)}-\frac{\hat{b}(t,s)}{H(t)}\right|^2ds\right.
				\\&\left.+\bar{\sigma}^4T\int_{0}^{t}\left|\frac{h^{n,\delta}(t,s)\hat{X}^{n,\delta}_P(s)}{H^{\delta}(t)}-\frac{h^{n}(t,s)\hat{X}^{n}(s)}{H(t)}\right|^2ds+\bar{\sigma}^4T\int_{0}^{T}\left|\frac{\hat{h}^{\delta}(t,s)}{H^{\delta}(t)}-\frac{\hat{h}(t,s)}{H(t)}\right|^2ds
				\right\}.
			\end{aligned}
		\end{equation}
		Then we will give the estimates for the above terms. By the definition of  $\{\tau_k^{\delta}\}_{k\in \mathbb{N}}$, we get
		\begin{equation}\label{5}
			|\hat{\phi}^{\delta}(t)-\hat{\phi}(t)|\leq\delta, \quad |\frac{1}{H^{\delta}(t)}-\frac{1}{H(t)}|\leq\delta,
		\end{equation}
		which implies that
		\[
		\begin{aligned}
			E_P\left[\left|\frac{\hat{\phi}^{\delta}(t)}{H^{\delta}(t)}-\frac{\hat{\phi}(t)}{H(t)}\right|^2\right]
			&\leq\frac{2}{m^2}E_P\left[\left|\hat{\phi}^{\delta}(t)-\hat{\phi}(t)\right|^2\right]+2E_P\left[\left|\hat{\phi}(t)\right|^2\left|\frac{1}{H^{\delta}(t)}-\frac{1}{H(t)}\right|^2\right]
			\\&\leq\delta^2\frac{2}{m^2}+2\delta^2\sup\limits_{t\leq T}\hat{\mathbb{E}}\left[|\hat{\phi}(t)|^2\right]<\infty.
		\end{aligned}
		\] 
		From the assumption (H3) and inequality \eqref{b_n_mo_contin}, we have for each $ (t,s)\in \Delta$,
		\begin{equation}\label{6}
			\begin{aligned}
				|b^{n,\delta}(t,s)-b^n(t,s)|=&\sum\limits_{k\geq0}|b^n(\tau_{k}^{\delta}\vee s,s)-b^n(t,s)|I_{[\tau_{k}^{\delta},\tau_{k+1}^{\delta})}(t) 
				\\\leq&2n\sum\limits_{k\geq0}\rho(|\tau_{k}^{\delta}\vee s-t|)I_{[\tau_{k}^{\delta},\tau_{k+1}^{\delta})}(t)\leq2n\rho(\delta)
			\end{aligned}
		\end{equation}
		and
		\begin{equation}\label{7}
			\begin{aligned}
				|\hat{b}^{\delta}(t,s)-\hat{b}(t,s)|=&\sum\limits_{k\geq0}|\hat{b}(\tau_{k}^{\delta}\vee s,s)-\hat{b}(t,s)|I_{[\tau_{k}^{\delta},\tau_{k+1}^{\delta})}(t)
				\\\leq&\sum\limits_{k\geq0}|b_1(\tau_{k}^{\delta}\vee s,s,\bar{X}_1(s)H(s))-b_1(t,s,\bar{X}_1(s)H(s))|I_{[\tau_{k}^{\delta},\tau_{k+1}^{\delta})}(t)
				\\&+\sum\limits_{k\geq0}|b_2(\tau_{k}^{\delta}\vee s,s,\bar{X}_1(s)H(s))-b_2(t,s,\bar{X}_1(s)H(s))|I_{[\tau_{k}^{\delta},\tau_{k+1}^{\delta})}(t)
				\\\leq&2\sum\limits_{k\geq0}\rho(|\tau_{k}^{\delta}\vee s-t|)I_{[\tau_{k}^{\delta},\tau_{k+1}^{\delta})}(t)\leq2\rho(\delta).
			\end{aligned}
		\end{equation}
		Similarly, we have $ |h^{n,\delta}(t,s)-h(t,s)| \leq2n\rho(\delta)$ and  $|\hat{h}^{\delta}(t,s)-\hat{h}(t,s)|\leq2\rho(\delta) $. Then, by a simple calculation, it follows from \eqref{6}-\eqref{7} that
		\[
		\begin{aligned}
			&E_P\left[\int_{0}^{t}\left|\frac{b^{n,\delta}(t,s)\hat{X}^{n,\delta}_P(s)}{H^{\delta}(t)}-\frac{b^{n}(t,s)\hat{X}^{n}(s)}{H(t)}\right|^2ds\right]
			\\\leq&4E_P\left[\int_{0}^{t}\frac{|b^{n,\delta}(t,s)|^2|\hat{X}^{n,\delta}_P(s)-\hat{X}^{n}(s)|^2+|b^{n,\delta}(t,s)-b^{n}(t,s)|^2|\hat{X}^{n}(s)|^2}{|H^{\delta}(t)|^2}ds\right]
			\\&+2E_P\left[\int_{0}^{t}|b^{n}(t,s)|^2|\hat{X}^{n}(s)|^2\left|\frac{1}{H^{\delta}(t)}-\frac{1}{H(t)}\right|^2ds\right]
			\\\leq&\frac{16L^2M^2}{m^2}\int_{0}^{t}E_P\left[|\hat{X}^{n,\delta}_P(s)-\hat{X}^n(s)|^2\right]ds+\left(\frac{16n^2\rho^2(\delta)}{m^2}+8L^2M^2\delta^2\right)E_P\left[\int_{0}^{T}|\hat{X}^{n}(s)|^2ds\right]
		\end{aligned}
		\]
		and
		\[
		\begin{aligned}
			&E_P\left[\int_{0}^{T}\left|\frac{\hat{b}^{\delta}(t,s)}{H^{\delta}(t)}-\frac{\hat{b}(t,s)}{H(t)}\right|^2ds\right]\\\leq&2E_P\left[\int_{0}^{T}\frac{|\hat{b}^{\delta}(t,s)-\hat{b}(t,s)|^2}{|H^{\delta}(t)|^2}ds\right]+2E_P\left[\int_{0}^{T}|\hat{b}(t,s)|^2\left|\frac{1}{H^{\delta}(t)}-\frac{1}{H(t)}\right|^2ds\right]
			\\\leq&\frac{8\rho^2(\delta)T}{m^2}+4\delta^2L^2\hat{\mathbb{E}}\left[\int_{0}^{T}(1+M^2|\bar{X}_1(s)|^2)ds\right].
		\end{aligned}
		\]
		Then we can derive the estimates for the last two terms of inequality \eqref{10} with respect to \( h \) similarly,.  Applying BDG's inequality \eqref{bdg}, we obtain 
		\[
		E_P\left[\left|\int_{0}^{t}\sigma^n(s)(\hat{X}^{n,\delta}_P(s)-\hat{X}^n(s))dB_s\right|^2\right]
		\leq 4\bar{\sigma}^2L^2M^2\int_{0}^{t}E_P\left[|\hat{X}^{n,\delta}(s)-\hat{X}^n(s)|^2\right]ds.
		\]
		According to the above estimates,  we conclude that
		\begin{equation}\label{11}
			\begin{aligned}
				E_P\left[|\hat{X}^{n,\delta}_P(t)-\hat{X}^n(t)|^2\right]\leq& C(\bar{\sigma},L,T)\left\{\left(\delta^2+\rho^2(\delta)\right)\frac{1}{m^2}+\left(\frac{n^2\rho^2(\delta)}{m^2}+M^2\delta^2\right)\hat{\mathbb{E}}\left[\int_{0}^{T}|\hat{X}^{n}(s)|^2ds\right]\right.
				\\&\left.+\delta^2\sup_{t\in[0,T]}\hat{\mathbb{E}}\left[|\hat{\phi}(t)|^2\right]+\delta^2\hat{\mathbb{E}}\left[\int_{0}^{T}(1+M^2|\bar{X}_1(s)|^2)ds\right]\right.	
				\\&\left.+\left(\frac{M^2}{m^2}+M^2\right)\int_{0}^{t}E_P\left[|\hat{X}^{n,\delta}_P(s)-\hat{X}^n(s)|^2\right]ds\right\},
			\end{aligned}
		\end{equation}
		where $C(\bar{\sigma},L,T)$ depends on $\bar{\sigma},L,T$. Combining \eqref{X_n<infty} with \eqref{X_ndelta<infty}, we have $E_P\left[|\hat{X}^{n,\delta}_P(t)-\hat{X}^n(t)|^2\right]<\infty.$
		Thus, applying Gronwall's inequality and  letting $\delta\rightarrow0$, we get
		\begin{equation}\label{12}
			\lim_{\delta\rightarrow0}E_P\left[|\hat{X}^{n,\delta}_P(t)-\hat{X}^n(t)|^2\right]=0,\quad n\in\mathbb{N}.
		\end{equation}
		Recalling Lemma \ref{nixianxing} and sending $n\rightarrow\infty$, we obtain \eqref{X_n_X}. By combining  \eqref{12} with  \eqref{X_n_X}, we conclude that
		\begin{equation}\label{two_approximation}
			\lim_{n\rightarrow\infty}\lim_{\delta\rightarrow0}E_P\left[|\hat{X}^{n,\delta}_P(t)-\hat{X}(t)|^2\right]=0.
		\end{equation}
		
		Finally, we will show that $ \hat{X}^{n,\delta}_P(t)\geq 0,\text{ }P\text{-}a.s.$ for each $ t\in[0,T] $, which will be proved by induction. The proof can be divided into two steps.\\
		Step \MakeUppercase{\romannumeral 1}: On interval $ [0,\tau_{1}^\delta) $, we have
		\begin{equation}\label{15}
			\begin{aligned}
				\hat{X}^{n,\delta}_P(t) &=\frac{\hat{\phi}(0)}{H(0)}+\int_{0}^{t}\frac{b^{n}(s,s)\hat{X}^{n,\delta}(s)+\hat{b}(s,s)}{H(0)}ds
				+\int_{0}^{t}\frac{h^{n}(s,s)\hat{X}^{n,\delta}_P(s)+\hat{h}(s,s)}{H(0)}d\langle B\rangle_s
				\\&+\int_{0}^{t}\sigma^n(s)\hat{X}^{n,\delta}_P(s)dB_s,
			\end{aligned}
		\end{equation}
		which is a SDE and $\hat{X}^{n,\delta}_P(\cdot)$ has continuous paths on $ [0,\tau_{1}^\delta) $. From (A1), (A2) and (A4), we have $\frac{\hat{\phi}(0)}{H(0)}\geq0$, $ \frac{b^{n}(s,s)}{H(0)}\ge0$ and $ \frac{h^{n}(s,s)}{H(0)}\geq0$.	Recalling the comparison theorem of SDE, it follows that 
		\begin{equation}\label{part1geq0}
			\hat{X}^{n,\delta}_P(t)\geq 0,\quad P\text{-}a.s.,  \text{ }\forall t \in [0,\tau_{1}^\delta)
		\end{equation}
		and 
		\begin{equation}\label{16}
			\begin{aligned}
				\hat{X}^{n,\delta}_P(\tau_{1}^\delta-0) =&\frac{\hat{\phi}(0)}{H(0)}+\int_{0}^{\tau_{1}^\delta}\frac{b^{n}(s,s)\hat{X}^{n,\delta}_P(s)+\hat{b}(s,s)}{H(0)}ds
				+\int_{0}^{\tau_{1}^\delta}\frac{h^{n}(s,s)\hat{X}^{n,\delta}_P(s)+\hat{h}(s,s)}{H(0)}d\langle B\rangle_s
				\\&+\int_{0}^{\tau_{1}^\delta}\sigma^n(s)\hat{X}^{n,\delta}_P(s)dB_s 
				\\\geq &0,  \text{ }P\text{-}a.s.
			\end{aligned}
		\end{equation}
		Step \MakeUppercase{\romannumeral 2}: On interval $ [\tau_{1}^\delta,\tau_{2}^\delta) $, we have
		\begin{equation}\label{17}
			\begin{aligned}
				\hat{X}^{n,\delta}_P(t) =&\frac{\hat{\phi}(\tau_{1}^\delta)}{H(\tau_{1}^\delta)}+\int_{0}^{t}\sigma^n(s)\hat{X}^{n,\delta}_P(s)dB_s+\int_{0}^{t}\frac{b^{n}(\tau_{1}^\delta\vee s,s)\hat{X}^{n,\delta}_P(s)+\hat{b}(\tau_{1}^\delta\vee s,s)}{H(\tau_{1}^\delta)}ds
				\\&+\int_{0}^{t}\frac{h^{n}(\tau_{1}^\delta\vee s,s)\hat{X}^{n,\delta}_P(s)+\hat{h}(\tau_{1}^\delta\vee s,s)}{H(\tau_{1}^\delta)}d\langle B\rangle_s.
			\end{aligned}
		\end{equation}
		Obviously, we get
		\begin{equation}\label{19}
			\begin{aligned}
				\hat{X}^{n,\delta}_P(t) =&\tilde{X}^{n,\delta}_P(\tau_{1}^\delta)+\int_{\tau_{1}^\delta}^{t}\sigma^n(s)\hat{X}^{n,\delta}_P(s)dB_s+\int_{\tau_{1}^\delta}^{t}\frac{b^{n}(s,s)\hat{X}^{n,\delta}_P(s)+\hat{b}(s,s)}{H(\tau_{1}^\delta)}ds
				\\&+\int_{\tau_{1}^\delta}^{t}\frac{h^{n}(s,s)\hat{X}^{n,\delta}_P(s)+\hat{h}(s,s)}{H(\tau_{1}^\delta)}d\langle B\rangle_s,
			\end{aligned}
		\end{equation}
		where
		\begin{equation}\label{18}
			\begin{aligned}
				\tilde{X}^{n,\delta}_P(\tau_{1}^\delta)
				=&\frac{\hat{\phi}(\tau_{1}^\delta)}{H(\tau_{1}^\delta)}-\frac{\hat{\phi}(0)}{H(0)}+\hat{X}^{n,\delta}_P(\tau_{1}^\delta-0)
				\\&+\int_{0}^{\tau_{1}^{\delta}}\left[\frac{b^{n}(\tau_{1}^{\delta} ,s)\hat{X}^{n,\delta}_P(s)+\hat{b}(\tau_{1}^{\delta},s)}{H(\tau_{1}^\delta)}-\frac{b^{n}(s,s)\hat{X}^{n,\delta}_P(s)+\hat{b}(s,s)}{H(0)}\right]ds
				\\&+\int_{0}^{\tau_{1}^{\delta}}\left[\frac{h^{n}(\tau_{1}^{\delta},s)\hat{X}^{n,\delta}_P(s)+\hat{h}(\tau_{1}^{\delta},s)}{H(\tau_{1}^\delta)}-\frac{h^{n}(s,s)\hat{X}^{n,\delta}_P(s)+\hat{h}(s,s)}{H(0)}\right]d\langle B\rangle_s.
			\end{aligned}
		\end{equation}
		Based on the assumptions (A1)-(A4) and Remark \ref{dandiao}, we know that \[
		\frac{\hat{\phi}(\tau_{1}^\delta)}{H(\tau_{1}^\delta)}\geq\frac{\hat{\phi}(0)}{H(0)},\quad \frac{l^{n}(\tau_{1}^{\delta},s)}{H(\tau_{1}^\delta)}\geq \frac{l^{n}(s,s)}{H(0)},\quad \frac{\hat{l}(\tau_{1}^{\delta},s)}{H(\tau_{1}^\delta)}\geq\frac{\hat{l}(s,s)}{H(0)},\quad l=b,h.
		\]
		Thus, combining \eqref{part1geq0} with \eqref{16}, we can derive $\tilde{X}^{n,\delta}_P(\tau_{1}^\delta)\geq 0. $  	Applying comparison theorem of SDE to \eqref{19} in a similar way, we  deduce
		\[
		\hat{X}^{n,\delta}_P(t)\geq 0,\quad P\text{-}a.s.,  \text{ }\forall t \in [\tau_{1}^\delta,\tau_{2}^\delta)
		\]
		
		In conclusion, by repeating the procedure outlined in Step \MakeUppercase{\romannumeral 1}-Step \MakeUppercase{\romannumeral 2}, we derive $ \hat{X}^{n,\delta}_P(t)\geq 0, \text{ }P\text{-}a.s.\text{ }\forall t\in [0,T]. $ Combining this result with \eqref{two_approximation}, we obtain that $\hat{X}(t)\ge0, \text{ }P\text{-}a.s.\text{ }\forall t\in[0,T], $ From the definition of capacity, we have 
		\[\hat{X}(t)\ge0, \quad q.s.,\text{ }\forall t\in[0,T].\] Recalling Corollary \ref{4jie_contin}, we can conclude that $\hat{X}(t)$ has continuous paths. Then we derive that 
		\[\hat{X}(t)\ge0, \quad \forall t\in[0,T],\text{ }q.s.\]
		The proof is complete.
	\end{proof}
	\begin{remark}
		By adopting the new quasilinearization approach ( see \eqref{def_n}) and a two-step approximation method (see \eqref{two_approximation}), the assumptions of the modulus of continuity, Lipschitz continuity for partial derivatives of coefficients, and the monotonicity of $H(t)$ are now unnecessarily restrictive, which are different from \cite{tudor1989a,ferreyra2000comparison,wang2015com}. Moreover, we do not need to assume the existence of partial derivatives of coefficients $\partial_{t}b_i, \partial_{x}b_i, i=1,2$.
	\end{remark}
	Now we establish the comparison theorem for $G$-SVIEs \eqref{4.1} under the condition that $\phi_i(\cdot)$ is mean-square continuous for $i=1,2$.
	\begin{theorem}\label{th4.3}
		Suppose that $  \phi_i(\cdot)\in \tilde{M}_{G}^{2}(0,T)  $ is mean-square continuous, $H$ satisfies (A1) and $ b_i, h_i, \sigma$ satisfy the assumptions (H1)-(H3), (A2)-(A4)  for $i=1,2$. Let $ X_i(\cdot),\ i=1,2 $ be the solutions of $G-$SVIEs \eqref{4.1} corresponding to $ \phi_i,b_i,h_i,\sigma$ and $H$. Then \[X_1(t)\geq X_2(t),  \quad q.s., \text{ }\forall t\in [0,T].		\]  
		
	\end{theorem}
	\begin{proof}
		The proof is similar to Theorem \ref{th4.2} and we will focus on the important parts. For each fixed $P\in \mathcal{P},\ k\in\mathbb{N},\ \delta>0$, define the sequence of stopping times $\{\tau_k^{\delta}\}_{k\in \mathbb{N}}$ similarly, 
		\[ 
		\tau_{k}^{\delta}=\inf\{t\in (\tau_{k-1}^{\delta},(\tau_{k-1}^{\delta}+\delta)\wedge T]: |\frac{1}{H(t)}-\frac{1}{H(\tau_{k-1}^{\delta})}|\ge\delta \} \wedge[(\tau_{k-1}^{\delta}+\delta)\wedge T],\quad 
		\]  where $\tau_{0}^{\delta}=0$	and $\inf \emptyset =\infty$. Then the sequence of stopping times $\{\tau_k^{\delta}\}_{k\in \mathbb{N}}$ satisfies $ 0=\tau_0^\delta<\tau_1^\delta<\tau_2^\delta<\cdots\leq T $ and $\max\limits_{1\leq k\leq N}|\tau_{k}^{\delta}-\tau_{k-1}^{\delta}| \leq\delta$. Then for each $\omega\in \Omega$, there exists	$m\in \mathbb{N}$ such that $\tau_{m}^{\delta}(\omega)=T.$
		Applying the same method as in \eqref{12} and \eqref{X_n_X}, we conclude that
		\[
		\lim_{n\rightarrow\infty}\lim_{\delta\rightarrow0}\sup_{t\in[0,T]}E_P\left[|\hat{X}^{n,\delta}_P(t)-\hat{X}(t)|^2\right]=0.
		\]
		The remaining proof proceeds similarly to that of Theorem \ref{th4.2}. In conclusion, we derive \[\hat{X}(t)\geq0,  \quad q.s., \text{ }\forall t\in [0,T].
		\]  
	\end{proof}
	\begin{remark}
		Since the integral \[H(t)\int_{0}^{t}\sigma(s,X_i(s))dB_s,\quad t\in[0,T],\]  is not mean-square continuous, the solution $X_i(\cdot)$ of G-SVIEs \eqref{4.1}  is  not mean-square continuous for $i=1,2$ under the conditions of Theorem \ref{th4.3}. On the other hand, as the solution $X_i$($\cdot$) of G-SVIEs \eqref{4.1} does not have continuous paths for $i=1,2$, we cannot derive that $\hat{X}(t)\ge0,\ \forall t\in[0,T],\text{ }q.s.$
	\end{remark}
	
	Now we  present an example of $G$-SVIEs. Note that the differentiability of the coefficients is not assumed in (A2)-(A4). The following example \ref{ex4.8} demonstrates that the conditions (A1)-(A4) are more suited for cases where partial derivatives do not exist, even when $G$ is a linear function.
	
	\begin{example}\label{ex4.8}
		Set $G(a)=\frac{1}{2}a$ for $a\in \mathbb{R}$. In this case, $B$ is a classical Brownian motion. Consider the $G$-SVIEs on [0,T],
		\begin{equation}\label{26}
			\begin{aligned}
				&X_1(t)=1+\int_{0}^{t}\left( -2(X_1(s)\vee 0)e^{-t}+1\right) ds+\int_{0}^{t}X_1(s)dB_s, 
				\\ &X_2(t)=1+\int_{0}^{t}\left(-2(X_2(s)\vee 0))e^{-t}\right) ds+\int_{0}^{t}X_2(s)dB_s,
			\end{aligned}
		\end{equation}
		where $b_1(t,s,x)=-2(x\vee0)e^{-t}+1$, $b_2(t,s,x)=-2(x\vee 0)e^{-t} $. It is easy to check that for each $y\geq x$, $(t,s)\in\Delta$, $b_1, b_2$ satisfy (A2)-(A3). By setting $y_0=e^{t}$, we obtain that for each $y_0=e^{t}>0>x_0$, $b_1(t,s,y_0)=-1<0=b_2(t,s,x_0)$. Then, it fails to satisfy the condition that for each $y\geq x$, $(t,s)\in\Delta$,
		\[
		b_1(t,s,y)\geq b_2(t,s,x), 
		\]
		which is proposed in \cite{tudor1989a,ferreyra2000comparison}. Moreover, the coefficients $b_1(t,s,x),b_2(t,s,x) $ are not differentiable with respect to $x$ at $x=0.$
	\end{example}

\end{document}